\documentclass[final,3p]{elsarticle}
\usepackage{amssymb}      
\usepackage{amsmath,amsthm}
\usepackage{graphicx}
\usepackage{mathrsfs}
\usepackage{enumitem}
\usepackage{xcolor}
\usepackage{bm}
\usepackage{accents}
\usepackage{multicol}
\usepackage{tikz}
\usetikzlibrary{cd}
\usepackage[cmtip,all]{xy}
\usepackage[colorlinks=true, linkcolor=blue, citecolor=blue, urlcolor=blue]{hyperref}
\usepackage{mathtools}

\newtheorem{thm}{Theorem}[section]
\newtheorem*{thm*}{Theorem}
\newtheorem{lem}[thm]{Lemma}
\newtheorem{prop}[thm]{Proposition}
\newtheorem{cor}[thm]{Corollary}
\theoremstyle{definition}
\newtheorem{defn}[thm]{Definition}

\newtheorem{terminology}[thm]{Terminology}
\theoremstyle{remark}
\newtheorem{remark}[thm]{Remark}
\numberwithin{equation}{section}

\newcommand{\R}{\mathbb{R}}  
  
\newcommand{\N}{\mathbb{N}}  
\newcommand{\rr}{\mathbb{R}^{2}_{<}}

\begin{document}

\begin{frontmatter}

\title{Equivalence of Landscape and Erosion Distances for Persistence Diagrams}

\author[fsu]{Cagatay Ayhan}
\ead{hca21a@fsu.edu}

\author[fsu]{Tom Needham}
\ead{tneedham@fsu.edu}

\address[fsu]{Department of Mathematics, Florida State University, Tallahassee, FL 32306, USA}

 \hbadness=10000

\begin{abstract}
This paper establishes connections between three of the most prominent metrics used in the analysis of persistence diagrams in topological data analysis: the bottleneck distance, Patel's erosion distance, and Bubenik's landscape distance. Our main result shows that the erosion and landscape distances are equal, thereby bridging the former's natural category-theoretic interpretation with the latter's computationally convenient structure. The proof utilizes the category with a flow framework of de Silva et al., and leads to additional insights into the structure of persistence landscapes. Our equivalence result is applied to prove several results on the geometry of the erosion distance. We show that the erosion distance is not a length metric, and that its intrinsic metric is the bottleneck distance. We also show that the erosion distance does not coarsely embed into any Hilbert space, even when restricted to persistence diagrams arising from degree-0 persistent homology. Moreover, we show that erosion distance agrees with bottleneck distance on this subspace, so that our non-embeddability theorem generalizes several results in the recent literature. 
\end{abstract}

\keywordtitle{2020 Mathematics Subject Classification}

\end{frontmatter}

\setlength{\emergencystretch}{3em}

\section{Introduction}

Persistent homology is the most prominent tool in the field of Topological Data Analysis (TDA). In its classical form, persistent homology analyzes a one-parameter family of topological spaces (typically arising as a multiscale representation of a dataset) and outputs a topological summary known as a \emph{persistence diagram}~\cite{zomorodian2004computing,carlsson2007theory}. Persistence diagrams take the form of multisets of points $(b,d) \in \R^2$ with $b < d$, each of which represents the lifespan of a homological feature in the input family of topological spaces. The space of persistence diagrams, which we denote $\mathsf{PDgm}$, can be endowed with a variety of metrics, including the \emph{bottleneck distance}~\cite{Cohen-Steiner2007}, \emph{Wasserstein distances}~\cite{cohen2010lipschitz}, \emph{erosion distance}~\cite{patel2018generalized}, and various (pseudo)metrics arising from embedding diagrams into Hilbert or Banach spaces for machine learning tasks (see~\cite{chazal2014stochastic,bubenik2015statistical,adams2017persistence,chung2022persistence}, among many others). Within the TDA community, there has been significant recent interest in exploring properties of these metrics~\cite{bubenik2020embeddings,wagner2021nonembeddability,mitra2021space,perea2023approximating,che2024basic}. 

This paper is concerned with connections between three of the most widely used metrics on $\mathsf{PDgm}$. First, the \emph{bottleneck distance}~\cite{Cohen-Steiner2007} is considered to be the canonical metric on $\mathsf{PDgm}$, due in part to its stability to noise in the input data and its natural interpretation in the language of category theory. Moreover, the bottleneck distance can be computed in polynomial time via combinatorial optimization (see Remark \ref{rmk:bottleneck_distance_computation}). Second, the \emph{landscape distance}~\cite{bubenik2015statistical} arises via a certain embedding of $\mathsf{PDgm}$ into a Banach space, called the \emph{landscape map}. This leads to a very easy-to-compute metric which is more amenable to statistical applications than the bottleneck distance, by performing them in the ambient Banach space.  However, it does not have an immediate category-theoretical interpretation, and the image of the Banach space embedding is not well understood. Finally, the \emph{erosion distance}~\cite{patel2018generalized} is an alternative metric on the space of persistence diagrams with category-theoretic origins. It provides a lower bound on bottleneck distance (hence enjoying the same stability to perturbations in the data), it has been suggested that it is efficient to compute (see the discussion surrounding \cite[Example 1]{xian2022capturing}), and it has found several uses in theory and applications, e.g.,~\cite{kim2021spatiotemporal,xin2023gril,kim2024interleaving,carriere2025sparsification}. 

Our main result shows that the latter two metrics discussed above are, in fact, the same:

\begin{thm*}[Theorem \ref{thm: d_E = Landscape Distance}]
    The landscape map is an isometry between the erosion distance and the landscape distance.
\end{thm*}

To prove the theorem, we adopt the \emph{category with a flow} framework developed by de Silva et al.~in \cite{de2018theory}, which involves realizing the erosion and landscape distances in this category-theoretic framework. While a more direct proof of this fact is possible---for instance, it can be derived from the more general result \cite[Theorem 4.3]{kim2024interleaving} (see Remark \ref{rmk:KMS_result})---we chose our particular approach because it sheds additional light on the structure of persistence landscapes. Indeed, in preparation for the proof of the main theorem, we provide in Theorem \ref{thm: Every landscape sequence is a persistence landscape of some persistence diagram} a novel characterization of the image of the landscape map. This result, in turn, utilizes a new algebraic structure on the space of persistence landscapes (Theorem \ref{thm: Landscape Decomposition Theorem}). 

The main theorem directly leads to several new results on the metric geometry of the erosion distance and its connections to bottleneck distance. Our main contributions in this direction are as follows:

\begin{itemize}
    \item It is known that the erosion distance is not equal to the bottleneck distance, in general---in fact, we show in Proposition \ref{prop:bottleck_and_erosion_not_coarsely_equivalent} that they are not even coarsely equivalent. However, we prove that they are locally equivalent; in particular, they generate the same topology on $\mathsf{PDgm}$ (Theorem \ref{thm: bottleneck is erosion}). In fact, we show that the bottleneck distance is generated as the intrinsic (geodesic) metric associated to the erosion distance (Theorem \ref{thm:Bottleneck is intrinsic to erosion}). Accordingly, bottleneck distance can be seen as the intrinsic metric associated to landscape distance. This answers questions posed in  \cite[Example 3.15]{kim2024interleaving}, on whether the erosion distance is geodesic.
    \item We restrict our attention to a specific subspace of persistence diagrams, namely, \textbf{birth-zero diagrams}, denoted $\mathsf{PDgm}^{0}$ (diagrams with points of the form $(0,d)$). We show that, on this subspace, the bottleneck and erosion (and hence the landscape) distances are equal (Theorem \ref{thm: d_B = d_E on PDgm^0}).
    \item There is a simple map taking $\mathsf{PDgm}^0$ into a Banach space; namely, the \emph{death-vectorization embedding}  $\mathrm{DV}: \mathsf{PDgm}^{0} \to \ell^{\infty}$, introduced by Patrangenaru et al.~\cite{patrangenaru2019challenges}, sorts the non-zero coordinates of the points in a diagram to obtain a finitely supported, non-increasing sequence. As a corollary to Theorem \ref{thm: d_B = d_E on PDgm^0}, we prove that this map is bi-Lipschitz  with respect to the bottleneck/erosion/landscape distances (Corollary \ref{cor: Death vectorization map is bilipschitz}). This refines the main stability result of \cite{memoli2023ephemeral} (see Remark \ref{rmk:relation_to_memoli_zhou}).
    \item Finally, we leverage our results on $\mathsf{PDgm}^{0}$ to show that this space, endowed with bottleneck/erosion/landscape distance, does not admit a coarse embedding into a Hilbert space. It follows that the full space $\mathsf{PDgm}$, endowed with any of the metrics we consider, also fails to admit such a coarse embedding (Theorem \ref{thm: space of persistence diagrams do not admit coarse embeddings}). Coarse embedding questions in TDA have seen quite a bit of activity in recent years, and our work generalizes the main results of~\cite{bubenik2020embeddings,mitra2021space}.
\end{itemize}

The structure of the paper is as follows. In Section \ref{Section: 2 Background}, we provide the reader with the required background material on TDA invariants and metrics. Section \ref{Section: 4 Erosion distance is landscape distance} is devoted to proving our main theorem, Theorem \ref{thm: d_E = Landscape Distance}, which says that erosion and landscape distances coincide. Our approach is to define categorical structures on the space of diagrams and on the space of landscapes, equip them with coflows, and show that the induced interleaving distances are equivalent, in the sense of \cite{de2018theory}. The main theorem is applied in Section \ref{section: Topological Results} to derive geometrical results on the erosion distance, where we shift focus on its connections to bottleneck distance and coarse embeddability questions.

\section{Background}\label{Section: 2 Background}

The focus of this paper is on connections between various invariants arising in the field of Topological Data Analysis (TDA). Throughout the paper, we assume that the reader is familiar with the basic concepts of TDA---see \cite{carlsson2014topological,dey2022computational} as general references. This section tersely recalls the main concepts of interest, and mostly serves to set  terminology and notation. With this terminology in place, we also recapitulate some of the main results of the paper in precise language.

\subsection{Topological Invariants of Data}

We begin by reviewing some of the main dataset descriptors studied in TDA.

\subsubsection{Persistence Modules} 

A \textbf{persistence module} is a functor 
\[
\mathcal{F}:(\mathbb{R}, \leq) \to \mathsf{vect}
\]
where $(\mathbb{R}, \leq)$ is the poset category of $\R$, and $\mathsf{vect}$ is the category of finite-dimensional vector spaces over a fixed field $\mathbb{F}$. This means that:
\begin{itemize}
    \item To each $r \in \mathbb{R}$, we assign a vector space $\mathcal{F}(r) \in \mathsf{vect}$.
    \item To each pair of real numbers $r \leq s$, we assign a linear map $\mathcal{F}(r \leq s): \mathcal{F}(r) \to \mathcal{F}(s)$.
\end{itemize}
These assignments satisfy the usual functorial properties: for any $r \in \mathbb{R}$, the map $\mathcal{F}(r \leq r)$ is the identity map $\mathcal{I}_{\mathcal{F}(r)}$ on $\mathcal{F}(r)$, and for any $r \leq s \leq t$, the composition rule $\mathcal{F}(r \leq t) = \mathcal{F}(s \leq t) \circ \mathcal{F}(r \leq s)$ holds.

In practice, persistence modules typically arise by applying the (degree-$k$) homology functor to a filtered simplicial complex; that is, to a functor from the poset category $(\R,\leq)$ into $\mathsf{SpCpx}$, the category of finite simplicial complexes with simplicial maps. For example, such a structure arises as the Vietoris-Rips complex of a finite metric space. A persistence module of this form is referred to as the \textbf{(degree-$k$) persistent homology} of the filtered simplicial complex.

As they are functors, the appropriate notion of a \textbf{morphism} between persistence modules $\mathcal{F}$ and $\mathcal{G}$ is a natural transformation. This amounts to a family of linear maps $(\varphi_r:\mathcal{F}(r) \to \mathcal{G}(r))_{r \in \R}$ such that, whenever $r \leq s$,
\[
\mathcal{G}(r \leq s) \circ \varphi_r = \varphi_s \circ \mathcal{F}(r \leq s).
\]
Then an \textbf{isomorphism} is a morphism $(\varphi_r)_r$ such that each $\varphi_r$ is an isomorphism of vector spaces. 

\subsubsection{Persistence Diagrams}\label{subsec:persistence_diagrams_background} 

A persistence module can be decomposed as a direct sum of 1-dimensional modules, each of which can be represented by a pair of (extended) real numbers $(b,d)$ with $-\infty \leq b < d \leq \infty$~\cite{crawley2015decomposition}. This decomposition gives rise to a \emph{persistence diagram}, or a multiset $\{(b_i,d_i)\}$ of pairs with $b_i < d_i$. In this paper, we restrict attention to persistence
modules whose interval decompositions consist of half-open intervals $[b,d)$, for the sake of convenience.

Persistence diagrams are the main dataset invariant used in TDA. When a diagram arises as a representation of persistent homology, each point $(b,d)$ is interpreted as describing the ``birth'' and ``death'' scales of a homological feature in the filtered simplicial complex. 

In this paper, we make simplifying assumptions that persistence diagrams are \emph{finite} multisets, whose points satisfy $-\infty < b < d < \infty$. This is partially for convenience and partially for technical reasons---see Remark \ref{rem:finiteness} for details. In any case, this aligns with the structure of persistence diagrams in computational applications, where one necessarily works with finite collections and where points with $d=\infty$ are either ignored or truncated to some fixed, large but finite, value.

Let $\rr = \{ (b,d)\in \R \times  \R  \mid b < d \}$ denote the set of points in $\mathbb{R}^{2}$ above the diagonal. With the conventions described above, a \textbf{persistence diagram} in our setting is a multiset of the form  $Y = \{(b_i,d_i)\}_{i=1}^N$, such that $(b_i,d_i) \in \rr$ for each $i$. 

\subsubsection{Rank Functions}\label{subsec:rank_function_background}

Let $Y = \{(b_i,d_i)\}_{i=1}^N$ be a persistence diagram. The associated \textbf{rank function} is the map 
\[
\mathrm{rank}[Y]: \R^2_{\leq} \to \mathbb{Z},
\]
where $\R^2_{\leq} \coloneqq \{(b,d) \in \R^2 \mid b \leq d\}$, defined by 
\begin{equation}\label{eqn:rank_function}
\mathrm{rank}[Y](b,d) \coloneqq |\{(b_i,d_i) \in Y \mid b_i \leq b \leq d < d_i \}|.
\end{equation}
In the setting of persistent homology, this function captures the number of topological features which persist over the interval $(b,d)$. Indeed, if $Y$ arises as a decomposition of a persistence module $\mathcal{F}$, then the rank can also be expressed as 
\begin{equation}\label{eqn:rank_function_module}
\mathrm{rank}[Y](b,d) = \mathrm{rank}(\mathcal{F}(b \leq d)),
\end{equation}
where the latter is the rank in the usual sense of linear algebra.

\begin{terminology}\label{terminology:properly_contains}
    We say that $(b_i,d_i) \in Y$ \textbf{properly contains} $(b,d) \in \R^2_{\leq}$ if $b_i \leq b$ and $d< d_i$. With this terminology in place, $\mathrm{rank}[Y](b,d)$ is the number of pairs $(b_i,d_i) \in Y$ that properly contain $(b,d)$.  
\end{terminology}

\subsubsection{Persistence Landscapes}\label{subsec:persistence_landscapes_background}

Persistence landscapes, introduced by Bubenik in \cite{bubenik2015statistical}, are a powerful representation of persistence diagrams. This stable representation has been shown to be both machine-learning friendly and statistically accessible \cite{bubenik2015statistical, bubenik2020persistence}, as it provides a method to map persistence diagrams (which are somewhat unwieldy from a classical statistics perspective) to vectors in a Banach space. 

Let $Y$ be a persistence diagram. As defined by Bubenik in \cite{bubenik2015statistical}, the \textbf{persistence landscape} of $Y$ is the function $\lambda^{Y}:\N \times \R \to \R$ defined by 
\[
\lambda^{Y}(k,t) = \sup\{h \geq 0 \mid \, \mathrm{rank}[Y](t-h,t+h)\geq k\}.
\]
Alternatively, $\lambda^{Y}$ can be viewed as a sequence of functions $\lambda^{Y} = (\lambda^{Y}_1, \lambda^{Y}_2,\dots)$, where $\lambda^{Y}_k \coloneqq \lambda^Y(k,\cdot)$ is called the \textbf{$k$-th landscape function}. As each landscape function is bounded, and only finitely many of them are not identically zero, $\lambda^Y$ lies in the Banach space $L^\infty(\mathbb{N} \times \R)$. 

In \cite{bubenik2020persistence}, Bubenik gives an equivalent definition of persistence landscapes in terms of the birth-death pairs, which we state below. First, for any $(b,d)\in \rr$, define the associated \textbf{tent function} $f_{(b,d)}:\R \to \R$ by
\begin{equation}\label{eqn:tent_function}
    f_{(b,d)}(t) =
    \begin{cases}
        0, \text{ if } t \notin (b,d) \\
        t-b, \text{ if } t\in (b,\frac{b+d}{2}] \\
        d-t, \text{ if } t\in (\frac{b+d}{2},d).
    \end{cases}
\end{equation}
We then state the equivalent definition from \cite{bubenik2020persistence} as a proposition.

\begin{prop}[\cite{bubenik2020persistence}]\label{prop: lambda Y is just kmax of envelope functions}
    Let $Y = \{ (b_i,d_i) \}_{i=1}^N$ be a persistence diagram. The associated $k$-th landscape function can be expressed as
    \[
    \lambda^{Y}_k(t) = k\mathrm{max}\{f_{(b_i,d_i)}(t) \, \mid \, i = 1,2,\dots,N\}
    \]
    where $k\mathrm{max}$ stands for the $k$-th maximum element, which we declare to be $0$ for $k > N$.
\end{prop}

One of the main contributions of this paper is an exact characterization of the image of the map $Y \mapsto \lambda^Y$ taking a diagram to its landscape---see Theorem \ref{thm: Every landscape sequence is a persistence landscape of some persistence diagram} and Corollary \ref{cor:image_of_landscape_map}. 

In Figure \ref{fig:persistence_diagram_rank_landscape}, we present an example of a persistence diagram $Y$, along with values of its rank function $\mathrm{rank}[Y]$, and its persistence landscape $\lambda^Y$. Note that, according to Theorem \ref{thm: Landscape Decomposition Theorem}, $\lambda^Y$ can be expressed as
\[
\lambda^Y = f_{(1,7)} \oplus f_{(3,8)} \oplus f_{(2,5)} \oplus f_{(2,5)} \oplus f_{(9,10)}
\]
where $\oplus$ denotes the direct sum operation defined in Definition \ref{defn: Direct sum operation}.

\begin{figure}
    \centering
    \includegraphics[width=0.9\linewidth]{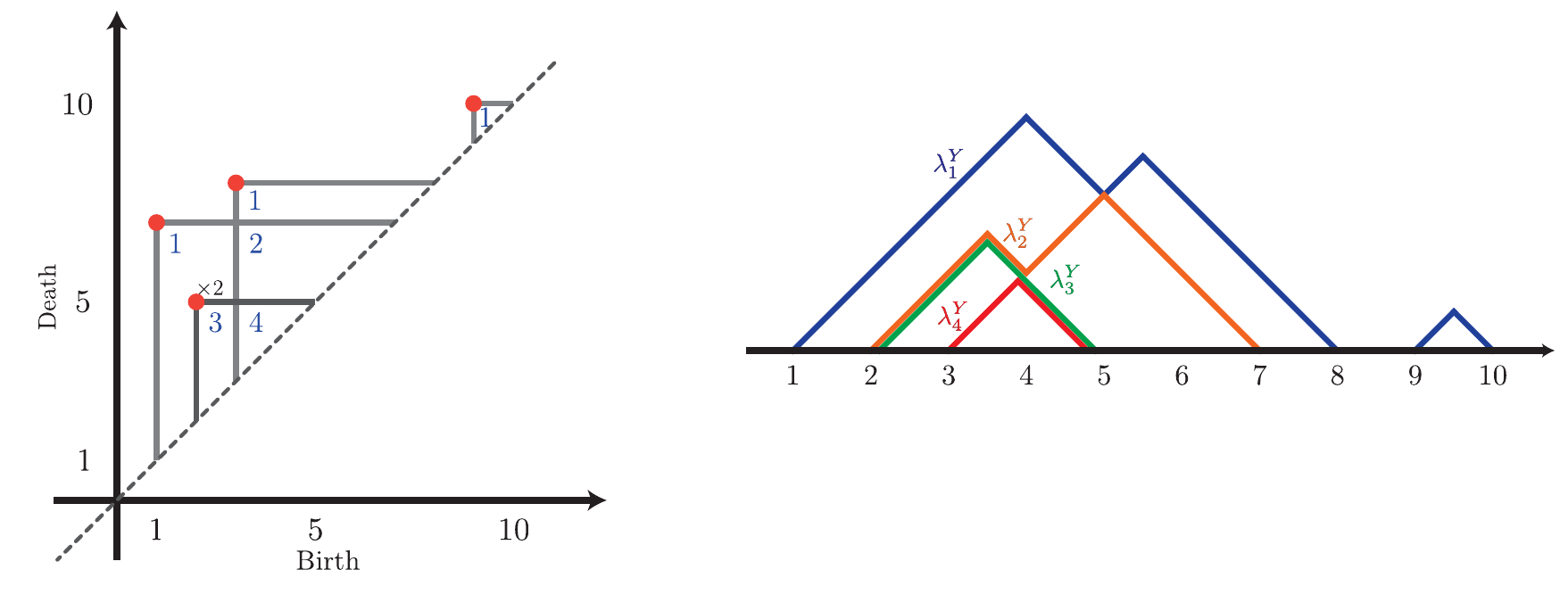}
    \caption{\textbf{Left}: An example persistence diagram given by $Y = \{ (1,7),(3,8),(2,5),(2,5),(9,10) \}$. The light-gray lines extending horizontally and vertically from each birth-death pair reveal the formation of the persistence landscape. Moreover, the values of $\mathrm{rank}[Y]$ are annotated in blue, associated with each region formed by the underlying landscape. \textbf{Right}: The corresponding persistence landscape $\lambda^Y$ which forms a non-increasing sequence of piecewise-linear functions $\lambda^Y_1 \geq \lambda^Y_2 \geq \lambda^Y_3 \geq \lambda^Y_4$ (and $\lambda^Y_k =0$ for $k \geq 5$), as indicated by different colors.
    }
\label{fig:persistence_diagram_rank_landscape}
\end{figure}

\subsection{Standard Metrics in TDA}

We now overview some metrics used to compare the TDA invariants described in the previous subsection. 

\subsubsection{Bottleneck Distance}\label{subsec:bottleneck_distance}

Let $Y = \{(b_i, d_i)\}_{i=1}^N$ and $Y' = \{(b_i', d_i')\}_{i=1}^{N'}$ be two persistence diagrams. The \textbf{bottleneck distance}~\cite{Cohen-Steiner2007} between $Y$ and $Y'$ is given by
\begin{align*}
    d_B(Y, Y') \coloneqq \min_{\phi: S \to S'} \max \bigg\{ &\max_{(b, d) \in S} d_\infty((b, d), \phi(b, d)), \max_{(b, d) \notin S} \frac{(d - b)}{2}, \max_{(b', d') \notin S'} \frac{(d' - b')}{2} \bigg\}
\end{align*}
where
\begin{itemize}
    \item $\phi: S \to S'$ is a  bijection with $S \subseteq Y$ and $S' \subseteq Y'$,
    \item $d_\infty((b, d), (b', d')) \coloneqq \max \{ |b - b'|, |d - d'| \}$.
\end{itemize}

Intuitively, $d_B(Y,Y')$ finds a matching between the topological features corresponding to points in the diagrams, which is optimal with respect to the $\ell^\infty$-type cost. The bottleneck distance is known to be stable to perturbations in the data from which the diagrams are derived~\cite{cohen2005stability}, and its computation is a tractable optimization problem; these properties, in part, lead to its status as the de facto canonical metric on the space of persistence diagrams.

\subsubsection{Interleaving Distance}\label{sec:interleaving_distance}

Loosely speaking, one can define a metric on the space of persistence modules by quantifying the degree to which any pair  is non-isomorphic; this motivates the concept of \emph{interleaving distance}, as introduced by Chazal et al.~\cite{chazal2009proximity}, which we now recall. For $\varepsilon\geq 0 $, two persistence modules $\mathcal{F}$ and $\mathcal{G}$ are said to be \textbf{$\varepsilon$-interleaved} if, for all $r\in \R$, there are linear maps $\varphi_{r}: \mathcal{F}(r) \to \mathcal{G}(r + \varepsilon)$ and $\psi_{r}: \mathcal{G}(r) \to \mathcal{F}(r + \varepsilon)$ such that the following diagrams commute for any $r \leq s$:

\[\begin{tikzcd}[ampersand replacement=\&,cramped,column sep=small]
	{\mathcal{F}(r)} \&\& {\mathcal{F}(s)} \&\& {\mathcal{F}(r)} \&\& {\mathcal{F}(r + 2\varepsilon)} \\
	\& {\mathcal{G}(r+\varepsilon)} \&\& {\mathcal{G}(s+\varepsilon)} \&\& {\mathcal{G}(r + \varepsilon)} \\
	\& {\mathcal{F}(r + \varepsilon)} \&\& {\mathcal{F}(s+\varepsilon)} \&\& {\mathcal{F}(r + \varepsilon)} \\
	{\mathcal{G}(r)} \&\& {\mathcal{G}(s)} \&\& {\mathcal{G}(r)} \&\& {\mathcal{G}(r + 2\varepsilon)}
	\arrow["{\mathcal{F}(r\leq s)}", from=1-1, to=1-3]
	\arrow["{\varphi_r}", from=1-1, to=2-2]
	\arrow["{\varphi_s}", from=1-3, to=2-4]
	\arrow["{\mathcal{F}(r \leq r + 2\varepsilon)}", from=1-5, to=1-7]
	\arrow["{\varphi_r}"', from=1-5, to=2-6]
	\arrow["{\mathcal{G}(r+\varepsilon \leq s +\varepsilon)}", from=2-2, to=2-4]
	\arrow["{\psi_{r + \varepsilon}}"', from=2-6, to=1-7]
	\arrow["{\mathcal{F}(r + \varepsilon \leq s+\varepsilon)}", from=3-2, to=3-4]
	\arrow["{\varphi_{r + \varepsilon}}", from=3-6, to=4-7]
	\arrow["{\psi_r}", from=4-1, to=3-2]
	\arrow["{\mathcal{G}(r \leq s)}", from=4-1, to=4-3]
	\arrow["{\psi_s}", from=4-3, to=3-4]
	\arrow["{\psi_r}", from=4-5, to=3-6]
	\arrow["{\mathcal{G}(r \leq r + 2\varepsilon)}", from=4-5, to=4-7]
\end{tikzcd}\]

The \textbf{interleaving distance} between $\mathcal{F}$ and $\mathcal{G}$ is then defined to be
\[
d_I(\mathcal{F},\mathcal{G}) = \inf\{\varepsilon \geq 0 \mid \mbox{$\mathcal{F}$ and $\mathcal{G}$ are $\varepsilon$-interleaved}\}.
\]

In general, the interleaving distance defines an extended pseudometric on the space of persistence modules. In fact, a fundamental theorem in TDA, referred to as the \emph{isometry theorem}~\cite{chazal2009proximity,lesnick2015theory}, says that the interleaving distance between persistence modules is equal to the bottleneck distance between their persistence diagram representations. 

The notion of interleaving distance has been extended over the years to define metrics for increasingly general structures \cite{Bubenik2014, Bubenik2015, de2018theory,stefanou2018dynamics,scoccola2020locally,mcfaddin2026interleaving}. An appropriate level of generality for our purposes is provided by the \emph{category with a flow framework} of de Silva et al.~\cite{de2018theory}, as we explain below in Section \ref{sec:coflows}.

\subsubsection{Erosion Distance}\label{subsec:erosion_distance_background}

Due to its stability properties and its connection to the interleaving distance (as described above), the bottleneck distance $d_B$ is considered to be the canonical metric on the space of persistence diagrams. However, the \emph{erosion distance}, introduced by Patel in \cite{patel2018generalized}, is also motivated by category theory and enjoys stability to perturbations in the data. We present a definition from \cite{xian2022capturing} here, and give Patel's (equivalent) definition below in Section \ref{subsec:erosion_distance}. The \textbf{erosion distance} between two persistence diagrams $Y$ and $Y'$ is given by

\begin{equation}\label{eqn:erosion_distance_initial}
\begin{split}
   d_E(Y,Y') \coloneqq \inf \bigg\{\varepsilon \geq 0 \mid \mathrm{rank}&[Y](b-\varepsilon,d+\varepsilon) \leq \mathrm{rank}[Y'](b,d) \mbox{ and } \\
   \mathrm{rank}&[Y'](b-\varepsilon,d+\varepsilon) \leq \mathrm{rank}[Y](b,d) \mbox{ for all } b\leq d\bigg\}. 
\end{split}
\end{equation}

It is not hard to show that $d_E \leq d_B$, nor is it hard to construct an example which illustrates that this is not an equality, in general (see, e.g.,  \cite[Example 1]{xian2022capturing}). However, we prove in Theorem \ref{thm: bottleneck is erosion} that these metrics are locally equivalent, and therefore induce the same topology on the space of diagrams. Moreover, we show in Theorem \ref{thm:Bottleneck is intrinsic to erosion} that the bottleneck distance is the intrinsic distance induced by the erosion distance.

\subsubsection{Landscape Distance}\label{subsec:landscape_distance}

As was observed in Section \ref{subsec:persistence_landscapes_background}, the persistence landscape $\lambda^Y$ of a diagram $Y$ lies in the Banach space $L^\infty(\mathbb{N} \times \mathbb{R})$. Landscapes are then naturally compared by the metric induced by the norm, referred to as the \textbf{landscape distance}:
\begin{equation}\label{eqn:landscape_distance_background}
\|\lambda^Y - \lambda^{Y'}\|_\infty \coloneqq \sup_{k,t} |\lambda^Y_k(t) - \lambda^{Y'}_k(t)|.
\end{equation}

The main result of this paper, Theorem \ref{thm: d_E = Landscape Distance}, shows that the landscape map $Y \mapsto \lambda^Y$ is an isometry with respect to the erosion and landscape distances. This result is leveraged to prove that the space of diagrams, endowed with erosion distance, does not coarsely embed into a Hilbert space (Theorem \ref{thm: space of persistence diagrams do not admit coarse embeddings}).

\subsection{Coflows on Poset Categories}\label{sec:coflows}

We will make heavy use of the \emph{category with a flow} framework of \cite{de2018theory}, which extends the interleaving distance construction (described in Section \ref{sec:interleaving_distance}) to define a very general class of metrics, presented in the language of category theory. For our purposes, it suffices to specialize the framework to the setting of poset categories. Moreover, \cite{de2018theory} distinguishes between \emph{lax} and \emph{strict} versions of the framework, whereas we only require the strict version. Finally, it turns out to be more convenient to use a related construction called a \emph{coflow}, as introduced in \cite{cruz2019metric}. For the sake of simplicity, we opt to present the framework specifically in the context of strict coflows on poset categories. For the rest of the paper, we frequently write $\mathsf{P} = (\mathsf{P},\leq)$ for a poset category---that is, we abuse notation and use $\mathsf{P}$ to denote both the category and the underlying set. We note that the following definitions could alternatively be expressed in the language of \emph{superlinear families of translations}, coined in~\cite{Bubenik2015}.

\begin{defn} \label{def: Coflow}
    Let $\mathsf{P} = (\mathsf{P},\leq)$ be a poset category. A family $\gamma = (\gamma_{\varepsilon})_{\varepsilon\geq 0}$ of endofunctors $\gamma_{\varepsilon}:\mathsf{P}\to \mathsf{P}$ is called a \textbf{coflow} on $\mathsf{P}$ if
    \begin{enumerate}
        \item $\gamma_{0} = \mathcal{I}_{\mathsf{P}}$, where $\mathcal{I}_{\mathsf{P}}$ is the identity endofunctor on $\mathsf{P}$;
        \item for all $\varepsilon \geq 0$ and $x \in \mathsf{P}$, $\gamma_\varepsilon(x) \leq x$; 
        \item $\gamma_{\varepsilon_{1} + \varepsilon_{2}} = \gamma_{\varepsilon_1} \circ \gamma_{\varepsilon_2}$ for all $\varepsilon_1,\varepsilon_2 \geq 0$.
    \end{enumerate}
    We will occasionally denote (by further abuse of notation) a poset category with a coflow as a triplet $\mathsf{P} = (\mathsf{P}, \leq, \gamma)$. In the following, for $x \in \mathsf{P}$ and $\varepsilon, \varepsilon_1,\varepsilon_2 \geq 0$, we write $\gamma_\varepsilon x$ in place of $\gamma_\varepsilon(x)$, and $\gamma_{\varepsilon_1}\gamma_{\varepsilon_2}$ in place of $\gamma_{\varepsilon_1} \circ \gamma_{\varepsilon_2}$.

    For $\varepsilon \geq 0$, we say that $x,y \in \mathsf{P}$ are \textbf{$\varepsilon$-interleaved (with respect to $\gamma$)} if $\gamma_\varepsilon x \leq y$ and $\gamma_\varepsilon y \leq x$. The \textbf{interleaving distance (associated to $\gamma$)} between $x$ and $y$, denoted $d_{\mathsf{P}}^{\gamma}(x,y)$, is the infimum of all $\varepsilon \geq 0$ such that $x$ and $y$ are $\varepsilon$-interleaved. That is,
    \[
     d_{\mathsf{P}}^{\gamma}(x,y) \coloneqq \inf
        \{
        \varepsilon \geq 0 \, \mid \, \gamma_{\varepsilon
        }x \leq y  \text{ and }  \gamma_{\varepsilon
        }y \leq x
        \}.
    \]
\end{defn}

\begin{remark}\label{rem:flows_and_coflows}
    A coflow on $(\mathsf{P},\leq)$ is a flow (in the language of \cite{de2018theory}) on the opposite poset category $(\mathsf{P},\leq_{\mathrm{op}})$, where $x \leq_\mathrm{op} y$ if and only if $x \geq y$.
\end{remark}

In light of Remark \ref{rem:flows_and_coflows}, the general result \cite[Theorem 2.7]{de2018theory} on interleaving distances associated to flows immediately implies the following.

\begin{prop}
    The interleaving distance $d_\mathsf{P}^\gamma$ associated to a coflow $\gamma$ defines an extended pseudometric on $\mathsf{P}$.
\end{prop}

Posets with coflows are related by the following construction.

\begin{defn}\label{def:coflow_equivariant}
    Let $(\mathsf{C},\leq_{\mathsf{C}},\tau)$ and $(\mathsf{D},\leq_{\mathsf{D}},\gamma)$ be two poset categories with coflows. A functor $\mathcal{H}: \mathsf{C} \to \mathsf{D}$ is called \textbf{coflow equivariant} if $\mathcal{H}\tau_{\varepsilon} = \gamma_{\varepsilon}
    \mathcal{H}$. If, in addition, $\mathcal{H}$ is an equivalence of categories, then $\mathcal{H}$ is called a \textbf{coflow equivariant equivalence} and we say that $\mathsf{C}$ and $\mathsf{D}$ are \textbf{coflow equivalent}. 
\end{defn}

A coflow equivariant equivalence defines an isometry of the associated interleaving distances. This is stated precisely in the following result, whose proof is straightforward; it also follows directly from the more general result \cite[Theorem 4.3]{de2018theory}.

\begin{prop} \label{prop: Coflow induced interleaving distances between coflow equivalent poset categories are preserved}
    Let $(\mathsf{C},\leq_{\mathsf{C}},\tau)$ and $(\mathsf{D},\leq_{\mathsf{D}},\gamma)$ be two poset categories with coflows and let $\mathcal{H}:\mathsf{C} \to \mathsf{D}$ be a coflow equivariant equivalence. Then, for any $a,b\in \mathsf{C}$, we have $d_{\mathsf{C}}^{\tau}(a,b) = d_{\mathsf{D}}^{\gamma}(\mathcal{H}(a),\mathcal{H}(b))$.
\end{prop}

\section{Category Structure of Persistence Diagrams and Landscapes} \label{Section: 4 Erosion distance is landscape distance}

This section is dedicated to developing formal category structures on spaces of objects which are fundamental to the theory of topological data analysis (TDA), introduced in Section \ref{Section: 2 Background}. 

\subsection{Persistence Diagrams} \label{Section: 3 Categorification}

Combining results and constructions which already exist in the literature, we now describe the space of persistence diagrams (see Section \ref{subsec:persistence_diagrams_background}) as a poset category, endowed with a coflow and associated interleaving distance, and show that this distance recovers Patel's \emph{erosion distance}~\cite{patel2018generalized} (see Section \ref{subsec:erosion_distance_background}). 

\subsubsection{Poset Category of Persistence Diagrams}

Recall that $\rr = \{ (b,d)\in \R \times  \R  \mid b < d \}$ denotes the set of points in $\mathbb{R}^{2}$ above the diagonal, and that $\R^2_\leq$ is the superset which also includes the diagonal. We frequently denote elements of $\mathbb{R}^2_{<}$ as $I=(b,d)$ (intended to evoke the idea that $I$ is an \emph{interval}), which we refer to as \textbf{birth-death pairs}, and refer to the first coordinate as the \textbf{birth} and the second coordinate as the \textbf{death} of $I$. We equip $\rr$ with the poset relation given by inclusion:
\[
(b,d) \subseteq (b',d') \iff b' \leq b \text{ and } d \leq d'.
\]
Finally, recall that a persistence diagram is a finite multiset of points in $\rr$, which we typically denote as $Y = \{(b_i,d_i)\}_{i=1}^N$ (mildly abusing notation and eschewing multiset notation). As was stated in the introduction, we denote the collection of all persistence diagrams as $\mathsf{PDgm}$.

\begin{remark}[Generalized Persistence Diagrams]
    Patel ~\cite{patel2018generalized} considered generalized persistence diagrams as maps $Y:\rr \to \mathsf{G}$ with finite support, where $\mathsf{G}$ is an abelian group. In the case that $\mathsf{G} = \mathbb{Z}$, and with the additional requirement that the map takes non-negative values, this is equivalent to our notion of a persistence diagram.
\end{remark}

\begin{remark}[Finiteness Assumptions]\label{rem:finiteness}
    One could study persistence diagrams with countably many points. For example, Bubenik and Wagner in ~\cite{bubenik2020embeddings} defined persistence diagrams as functions $D: \mathcal{I} \to \rr$, where $\mathcal{I}$ is a countable indexing set. Throughout this paper, we define persistence diagrams to be finite, following ~\cite{patel2018generalized}, because the definition of Patel's erosion distance relies on rank functions, which are integer-valued. 
    
    Frequently in TDA, one considers persistence diagrams where death coordinates are allowed to take the value $\infty$. Points with this property represent topological features which never vanish in the filtration. We explicitly work with finite birth-death values rather than allowing birth-death pairs with infinite persistence. This assumption is only for convenience, as it avoids the need for cumbersome case distinctions in various definitions and results. Nevertheless, most of the main results remain valid if one chooses to include birth-death pairs with infinite persistence, with only some adjustments to the definitions that follow.
\end{remark}

A poset structure on $\mathsf{PDgm}$ is defined in terms of the rank function (see Section \ref{subsec:rank_function_background}) as follows.

\begin{defn}[\cite{patel2018generalized}]
    We define the \textbf{persistence diagram poset structure}  $\leq_{\mathsf{PDgm}}$ on $\mathsf{PDgm}$ as follows:
    \[
    Y \leq_{\mathsf{PDgm}} Y' \iff \mathrm{rank}[Y](I) \leq \mathrm{rank}[Y'](I) \text{ for all } I \in \R^2_\leq.
    \]
    As the context is always clear, we simplify the notation and write $Y \leq Y'$ instead of $Y \leq_{\mathsf{PDgm}}Y'$.
\end{defn}

\subsubsection{Erosion Distance is a Coflow-Induced Interleaving Distance}\label{subsec:erosion_distance}

We now reformulate the erosion distance \eqref{eqn:erosion_distance_initial} in more categorical terms, in order to take advantage of the coflow framework of Section \ref{sec:coflows}. This follows the original construction of Patel, translated into our notation.

\begin{defn}[\cite{patel2018generalized}]
    Given $\varepsilon\geq 0$, we define the \textbf{$\varepsilon$-grow map} $\mathrm{Grow}_{\varepsilon}:\R_{\leq}^2 \to \R_{\leq}^2$ by
    \[
    \mathrm{Grow}_{\varepsilon}(b,d) \coloneqq (b - \varepsilon,d + \varepsilon).
    \]
    The $\varepsilon$-grow map induces a functor $\nabla_\varepsilon:\mathsf{PDgm} \to \mathsf{PDgm}$ defined by 
    \[
    \nabla_\varepsilon Y \coloneqq \{(b,d) \in \rr  \mid \mathrm{Grow}_\varepsilon(b,d) \in Y\}.
    \]
The \textbf{erosion distance}, $d_E$, between $Y$ and $Y'$ is defined by
\[
d_{E}(Y,Y') \coloneqq \inf\{\varepsilon\geq 0\,|\, \nabla_\varepsilon Y' \leq Y \text{ and } \nabla_\varepsilon Y\leq Y' \}.
\]
\end{defn}

We will now show that the family $\nabla = \{ \nabla_{\varepsilon} \}_{\varepsilon \geq 0}$ is a coflow and that the corresponding interleaving distance is the erosion distance. This perspective on erosion distance was already explored by Puuska in~\cite{puuska2020erosion}, albeit in the language of superlinear families of translations~\cite{Bubenik2015}, rather than that of flows on categories~\cite{de2018theory}, and by Kim, M\'{e}moli and Stefanou in~\cite{kim2024interleaving}. We provide the details in our case of interest, for the convenience of the reader. The first step is to prove the following.

\begin{prop}\label{Delta is a coflow}
    The family $\nabla = (\nabla_{\varepsilon})_{\varepsilon \geq 0}$ defines a coflow on $\mathsf{PDgm}$.
\end{prop}

The proof will use the following results; the first two are obvious, while the last follows from \cite[Proposition 5.1]{patel2018generalized},  but is proved here, in order to mitigate difficulties in translating between notational conventions.

\begin{lem} \label{lem: Grow e1 o Grow e2 = Grow e1 + e2}
    For $\varepsilon_1,\varepsilon_2 \geq 0$, we have that $\mathrm{Grow}_{\varepsilon_1}\circ \mathrm{Grow}_{\varepsilon_2} = \mathrm{Grow}_{\varepsilon_1 + \varepsilon_2}$.
\end{lem}

\begin{lem}\label{rankY is order preserving}
    The rank function $\mathrm{rank}[Y]$ is order-reversing. That is, if $I_1 \subseteq I_2$, then 
    \[
    \mathrm{rank}[Y](I_2)\leq \mathrm{rank}[Y](I_1).
    \]
\end{lem}

\begin{lem}[{\cite[Proposition 5.1.]{patel2018generalized}}] \label{prop 5.1}
    Let $Y$ be a persistence diagram. Then, given any $\varepsilon\geq 0$, we have $\mathrm{rank}[Y]\left(\mathrm{Grow}_{\varepsilon}(I) \right)= \mathrm{rank}[\nabla_\varepsilon Y](I)$ for all $I \in \mathbb{R}_{\leq}^2$.
    \end{lem}

\begin{proof}
    For $I=(b,d) \in \mathbb{R}_{\leq}^2$, we have 
    \begin{align*}
    & \mathrm{rank}[Y]\left(\mathrm{Grow}_{\varepsilon}(I)\right) \\
    &\qquad = |\{(b_i,d_i) \in Y \mid b_i \leq b-\varepsilon \leq d + \varepsilon < d_i \}| \\
    &\qquad = |\{(b_i,d_i) \in Y \mid b_i + \varepsilon \leq b \leq d < d_i - \varepsilon \}| \\
    &\qquad = |\{(b_i + \varepsilon,d_i - \varepsilon) \in \R_{\leq}^2 \mid (b_i,d_i) \in Y, \; b_i + \varepsilon \leq b \leq d < d_i - \varepsilon \}| \\
    &\qquad = |\{(b_i',d_i') \in \R_{\leq}^2 \mid \mathrm{Grow}_\varepsilon(b_i',d_i') \in Y, \; b_i' \leq b \leq d < d_i' \}|\\
    &\qquad = \mathrm{rank}[\nabla_\varepsilon Y](I).
    \end{align*}
\end{proof}

\begin{proof}[Proof of Proposition \ref{Delta is a coflow}]
    Let $Y \in \mathsf{PDgm}$. First, observe that 
    \[
    \nabla_0 Y = \{(b,d) \in \rr \mid \mathrm{Grow}_0(b,d) \in Y\} = Y,
    \]
    so that the first axiom in Definition \ref{def: Coflow} is satisfied. To prove that the second axiom holds, observe that, by Lemmas \ref{rankY is order preserving} and \ref{prop 5.1}, we have, for all $I \in \R_\leq^2$,
    \[
    \mathrm{rank}[\nabla_\varepsilon Y](I) = \mathrm{rank}[Y](\mathrm{Grow}_{\varepsilon}(I)) \leq \mathrm{rank}[Y](I).
    \]
    Thus $\nabla_\varepsilon Y \leq Y$. Finally, let $\varepsilon_1, \varepsilon_2 \geq 0$ and apply Lemma \ref{lem: Grow e1 o Grow e2 = Grow e1 + e2} to deduce that 
    \[
    \begin{split}
       \nabla_{\varepsilon_1} \nabla_{\varepsilon_2} Y &= \{(b,d) \in \rr \mid \mathrm{Grow}_{\varepsilon_1} \circ \mathrm{Grow}_{\varepsilon_2} (b,d) \in Y\} \\
       &= \{(b,d) \in \rr \mid \mathrm{Grow}_{\varepsilon_1 + \varepsilon_2}  (b,d) \in Y\} = \nabla_{\varepsilon_1 + \varepsilon_2} Y. 
    \end{split}
    \]
    This proves that the final axiom of a coflow is satisfied.
\end{proof}

This leads immediately to the main result of this subsection.

\begin{prop} \label{prop: The interleaving distance induced by the coflow nabla is erosion distance}
    The interleaving distance induced by the coflow $\nabla$ is the erosion distance.
\end{prop}

\begin{proof}
    This is clear by definition, as  
    \[
    d_{\mathsf{PDgm}}^{\nabla}(Y,Y') = \inf\{
    \varepsilon \geq 0 \, | \, \nabla_\varepsilon Y\leq Y' \text{ and } \nabla_\varepsilon Y' \leq Y 
    \} 
    = d_{E}(Y,Y').
    \]
\end{proof}

\subsection{Persistence Landscapes} \label{subsection: Persistence landscapes}

We introduced the persistence landscape $\lambda^Y$ of a persistence diagram in Section \ref{subsec:persistence_landscapes_background}. In this section, we derive a characterization of the image of the map $Y \mapsto \lambda^Y$ as a subset of $L^\infty(\mathbb{N} \times \mathbb{R})$. This is used in Section \ref{subsec:coflow_on_landscapes} to define a categorical structure on this subspace and to then interpret the landscape distance \eqref{eqn:landscape_distance_background} in terms of a coflow.

\subsubsection{Persistence Landscapes and Landscape Sequences}

Our first goal is to define landscapes as abstract objects, apart from arising solely as invariants of persistence diagrams. We begin with a preliminary definition.

\begin{defn}\label{def:landscape_curve}
    A function $\R \to \R$ is called a \textbf{landscape curve} if:
    \begin{enumerate}
        \item It is non-negative, continuous and piecewise-linear.
        \item It is compactly supported.
        \item It is differentiable except at finitely many points, called \textbf{critical points}.
        \item At non-critical support points, its derivative is $+1$ or $-1$. 
    \end{enumerate}
    These conditions imply that any critical point must be a local maximum or local minimum (exclusively) of the function.
\end{defn}

We now define our abstract notion of a persistence landscape, referred to as a \emph{landscape sequence}. 

\begin{defn} \label{def: landscape sequence}
    A \textbf{landscape sequence} is a sequence $\lambda = (\lambda_1,\lambda_2,\dots)$ of functions $\lambda_k:\R \to \R$ satisfying: 
    \begin{enumerate} 
        \item (Structure of $\lambda_k$) Each $\lambda_k$ is a landscape curve.
        \item (Monotonicity) For all $k\in \mathbb{N}$ and $t\in \mathbb{R}$, $\lambda_{k}(t) \geq \lambda_{k+1}(t)$.
        \item (Finite support) There exists $K \in \mathbb{N}$ such that $\lambda_k = 0$ for all $k \geq K$.
        \item (Non-negative degree) Any point $(t,h) \in \mathbb{R}^2$ with $h > 0$ has non-negative degree,  where the \textbf{degree} of a point $(t,h)$ is defined by
        \[
        \operatorname{deg}_{\lambda}(t,h) = |\{k\mid (t,h) \mbox{ is a local max. of } \lambda_k\}| - |\{k\mid (t,h) \mbox{ is a local min. of } \lambda_k\}|.
        \]
    \end{enumerate}
\end{defn}

\begin{terminology}
The term \textbf{critical point} may refer either to a pair $(t,h)$ or to the coordinate $t$ alone; i.e., when the height $h$ is implicit or can be inferred from the function values, we may simply refer to $t$ as a critical point. The set of \textbf{critical points} of the landscape sequence $\lambda$ is defined to be the union of the sets of critical points of each $\lambda_k$.
\end{terminology}

With this abstract definition of landscape sequences in hand, the main result of this subsection is that landscape sequences are, in fact, in one-to-one correspondence with persistence landscapes (arising from a persistence diagram).

\begin{thm}[Structure of Persistence Landscapes]\label{thm: Every landscape sequence is a persistence landscape of some persistence diagram}
    The map $\mathsf{PDgm} \to L^\infty(\mathbb{N} \times \R)$ given by $Y \mapsto \lambda^Y$ defines a bijection between the set of persistence diagrams and the set of landscape sequences. That is:
    \begin{enumerate}
        \item Given $Y \in \mathsf{PDgm}$, its persistence landscape $\lambda^Y$ is a landscape sequence.
        \item Conversely, every landscape sequence $\lambda$ is a persistence landscape of a unique persistence diagram $Y\in \mathsf{PDgm}$.
    \end{enumerate} 
\end{thm}

The ``forward'' direction, part 1, is relatively straightforward, via structural properties of landscape functions that were established in~\cite[Section 5]{betthauser2022graded}. The converse, part 2, is less clear, and is the main focus of the proof. An immediate corollary is that we can explicitly describe the image of the persistence landscape map. As far as we are aware, such a description has not been previously established in the literature.

\begin{cor}\label{cor:image_of_landscape_map}
    The image of the landscape map $Y \mapsto \lambda^{Y}$ is exactly the set of landscape sequences, which is a proper, nonlinear (moreover, non-affine linear) subset of the space of functions $L^{\infty}(\N \times \R)$.
\end{cor}

The proof of Theorem \ref{thm: Every landscape sequence is a persistence landscape of some persistence diagram}  will be derived from a preliminary result, concerning an algebraic description of the structure of landscape sequences. This description and the ensuing proof of Theorem \ref{thm: Every landscape sequence is a persistence landscape of some persistence diagram} are provided in the following subsection.

\subsubsection{Decomposition Theorem for Landscape Sequences}

To describe the structure of the space of landscape sequences, we introduce the following algebraic operation. It uses the \emph{kmax function}, introduced in Proposition \ref{prop: lambda Y is just kmax of envelope functions}.

\begin{defn} \label{defn: Direct sum operation}
    Given two landscape sequences $\eta$ and $\mu$, we define their \textbf{direct sum} 
    \[
    \lambda = \eta\oplus \mu
    \]
    to be a sequence of functions $\lambda_k$ given by
    \[
    \lambda_k(t) = k\mathrm{max}\{\eta_{1}(t),\dots,\eta_{k}(t),\mu_{1}(t),\dots,\mu_{k}(t)
    \}.
    \]
\end{defn}

\begin{prop} \label{prop: Degree nonzero prop}
    The sequence of functions defined by $\lambda = \eta \oplus \mu$ is also a landscape sequence.
\end{prop}

\begin{proof}
    Let $\lambda = \eta \oplus \mu$, where $\eta$ and $\mu$ are two landscape sequences. We will omit the verification of Properties 1, 2, and 3 in Definition \ref{def: landscape sequence}, as they follow directly from the construction of the direct sum operation and are relatively straightforward to verify. Our goal is to show that $\lambda$ satisfies Property 4.  
    
    Let $\{(t_\ell,h_\ell)\}$ denote the multiset of local minima of $\lambda$ with positive second coordinates, $h_\ell > 0$. Such a critical point $(t_\ell,h_\ell)$ of $\lambda_k$ can arise in exactly one of four ways:
    \begin{enumerate}
        \item as a local minimum of some $\eta_i$;
        \item as a local minimum of some $\mu_i$;
        \item there are $\eta_i$, with negative slope at $t_\ell$, and $\mu_j$, with positive slope at $t_\ell$, such that, for $t$ in a sufficiently small neighborhood of $t_\ell$, $\lambda_k(t) = \eta_i(t)$ for $t < t_\ell$ and $\lambda_k(t) = \mu_j(t)$ for $t > t_\ell$; or,
        \item the same as the above, with the roles of $\eta_i$ and $\mu_j$ switched: there are $\eta_i$ (positive slope at $t_\ell$) and $\mu_j$ (negative slope at $t_\ell$) such that, for $t$ in a sufficiently small neighborhood of $t_\ell$, $\lambda_k(t)=\mu_j(t)$ for $t<t_\ell$ and $\lambda_k(t)=\eta_i(t)$ for $t>t_\ell$.
    \end{enumerate}
    The collection $\{(t_\ell,h_\ell)\}$ can be decomposed into four disjoint (possibly empty) multisets, according to the four possibilities described above. We will show that for each of these multisets, there is an injective map into the multiset of local maxima for $\lambda$. This means that, at each critical point, the number of times that it is a local maximum must be greater than or equal to the number of times that it is a local minimum; that is, its degree is non-negative. 

    Consider the sub-multiset of local minima arising as in case 1 (local minima of $\eta$). Since $\eta$ is itself a landscape sequence, there must be an injective map into the multiset of local maxima of $\eta$. We now observe that each such local maximum of $\eta$ (i.e., those that are at a local minimum of $\lambda$) injectively maps into a local maximum of $\lambda$. Indeed, let $m$ be the maximum index such that $\lambda_m(t_\ell) = h_\ell$. It is not hard to show that $\lambda_m$ forms a local maximum at $(t_\ell,h_\ell)$. Moreover, if $\eta$ has another such local maximum, then we can let $m_2$ be the second-largest index such that $\lambda_{m_2}(t_\ell) = h_\ell$; this too forms a local maximum at $(t_\ell,h_\ell)$, verifying injectivity. So we have constructed the desired injective map. The map is constructed similarly for case 2.

    For each local minimum $(t_\ell,h_\ell)$ arising as in case 3, there is an associated local maximum described as follows (using the same indices from the description of case 3 above). For some $m > k$, there is a local maximum of $\lambda_m$ at $(t_\ell,h_\ell)$ where, for $t$ in a sufficiently small neighborhood of $t_\ell$, $\lambda_m(t) = \mu_j(t)$ for $t < t_\ell$ and $\lambda_m(t) = \eta_i(t)$ for $t > t_\ell$. This correspondence between local minima and associated local maxima gives the desired injective map, and case 4 is similar.
\end{proof}

Proposition \ref{prop: Degree nonzero prop} says that direct sum gives a well-defined binary operation on the space of landscape sequences. Next, we claim that any landscape sequence $\lambda$ can be expressed as a direct sum of tent functions, which are the simplest landscape sequences. To clarify this statement, let $b < d$ with associated tent function $f_{(b,d)}$ (see \eqref{eqn:tent_function}). By abuse of notation we denote the \textbf{associated landscape sequence} by $f_{(b,d)}$, where this really means the sequence of functions
\[
(f_{(b,d)},0,0,\ldots). 
\]

\begin{thm}[Landscape Decomposition Theorem]\label{thm: Landscape Decomposition Theorem}
    Let $\lambda$ be a landscape sequence. There exists a finite multiset of tent functions $\{ f_{(b_i,d_i)} \}$ such that 
    \[
    \lambda = \bigoplus_{i} f_{(b_i,d_i)},
    \]
    where each summand $f_{(b_i,d_i)}$ appears with multiplicity
    \[
    \mathrm{deg}_{\lambda}\bigg(\frac{b_i+d_i}{2} , \frac{d_i - b_i}{2}\bigg)>0.
    \]
    Moreover, this decomposition is unique up to ordering of summands.
\end{thm}

Before we prove Theorem \ref{thm: Landscape Decomposition Theorem}, we provide a lemma observing the effect of removing a particular tent function from a landscape sequence. In the lemma, and through the rest of the paper, for a landscape curve $\lambda_i$, we use $D^- \lambda_i(t)$ (respectively, $D^+ \lambda_i(t)$) to denote the \textbf{derivative of $\lambda_i$ from the left} (respectively \textbf{right}) at the point $t$; this is well-defined at every point for a landscape curve.

\begin{lem}\label{lem: Landscape Decomposition Theorem Lemma}
    Let $\lambda$ be a landscape sequence, and assume that $(x_1,y_1)$ is the left-most point where $\lambda_1$ attains a local maximum. Consider points
    \[
    x_1 \leq x_2 \leq x_3 \leq \dots \leq x_n 
    \]
    where each $x_i$, with $1\leq i \leq n $, is paired with $y_i := \lambda_i(x_i)$ such that\footnote{These are the points in $\lambda$ that intersect the downward leg of the tent function $f_{(x_1-y_1,x_1+y_1)}$. It may happen that the only such point is $(x_1,y_1)$, corresponding to the case $n=1$. See Figure \ref{fig:removing_tent_from_landscape_sequence}.}
    \[
        f_{(x_1-y_1,x_1+y_1)}(x_i) = \lambda_{i}(x_i)= y_i \quad \mbox{and} \quad D^{-}\lambda_i(x_i) = +1.
    \]
    Moreover, let $x_{j} = x_1 + y_1$ for $j \geq n+1$. Define the following sequence of functions $\eta$ as follows
    \begin{align*}
    \eta_k(t) = 
    \begin{cases}
        \lambda_{k+1}(t) & \text{if } t < x_{k+1} \\
        \lambda_{k}(t) & \text{if } x_{k+1}\leq t
        \end{cases}
    \end{align*}
    for each $k\in \N$. Then: 
    \begin{enumerate}
        \item $\eta$ is a landscape sequence.
        \item $\eta$ has strictly fewer local maxima than $\lambda$.
        \item $\lambda = f_{(x_1-y_1,x_1+y_1)} \oplus \eta$.
    \end{enumerate}
\end{lem}

\begin{proof}
    \noindent\textbf{Item 1.} We begin by proving the first assertion of the lemma.
    It is rather straightforward to verify that $\eta$ satisfies Properties $1$, $2$ and $3$ in Definition \ref{def: landscape sequence}. We will focus on verifying Property $4$.

    Let $(t_\ell,h_\ell)$ be a critical point in $\eta$, with $h_\ell >0$. To prove that the degree at this point is non-negative, we aim to show that each $\eta_k$ that attains a local minimum at $(t_\ell,h_\ell)$ is uniquely paired with some $\eta_j$ that attains a local maximum at $(t_\ell,h_\ell)$. We begin with the following observation:
    
    \paragraph{Claim} If $\eta_k$ attains a local minimum at $(t_\ell,h_\ell)$, then either $t_\ell < x_{k+1}$ or $t_\ell > x_{k+1}$; in other words, $t_\ell \neq x_{k+1}$.

    To prove the claim, we argue by contradiction. Suppose that $t_\ell = x_{k+1}$. Then, for $t^-$ sufficiently close to $t_\ell = x_{k+1}$ with $t^-<t_\ell$ , we have $\eta_k(t^-) = \lambda_{k+1}(t^-)$. Since $\eta_k$ attains a local minimum at $(t_\ell,h_\ell)$, this implies that $\lambda_{k+1}$ must have negative slope near such $t^-$. 
    
    Now, consider two possibilities: either $x_{k+1} < x_{n+1}$ or $x_{k+1} = x_{n+1} = x_1 + y_1$. If $x_{k+1}<x_{n+1}$, then by definition of $x_{k+1}$, $\lambda_{k+1}$ must have positive slope for $t^- < x_{k+1} = t_\ell$ sufficiently close to $t_\ell$. This contradicts the earlier conclusion that $\lambda_{k+1}$ has negative slope. Hence, we must have $x_{k+1} = x_{n+1}$.

    In this case, since $\eta_k$ is a local minimum curve and satisfies $\eta_k(t^{-}) = \lambda_{k+1}(t^{-})$ for $t^{-}<x_{k+1}$, the slope of $\lambda_{k+1}$ near $t^{-}$ must again be negative. But this would imply that $x_{k+1}<x_{n+1}$, since $\lambda_{k+1}$ would then be forced to intersect the tent function $f_{(x_1-y_1,x_1+y_1)}$ at a point earlier than $x_{n+1}$. This contradicts the earlier established fact that $x_{k+1} = x_{n+1}$.
    
    Therefore, both possibilities, $x_{k+1} < x_{n+1}$ and $x_{k+1} = x_{n+1}$, lead to contradictions. We conclude that $t_\ell \neq x_{k+1}$; i.e., either $t_\ell<x_{k+1}$ or $x_{k+1}<t_\ell$. This proves the claim. We now study these two cases. We first show that $\eta_k$ is paired with some $\eta_j$ that attains a local maximum at $(t_\ell,h_\ell)$; we then argue that this pairing is unique.
    
    \paragraph{Case 1} Suppose $t_\ell < x_{k+1}$. Then, $\eta_k = \lambda_{k+1}$ in a sufficiently small neighborhood of $t_\ell$. This equality implies that $(t_\ell,h_\ell)$ is a local minimum point of $\lambda_{k+1}$. Since $\lambda$ has non-negative degree, $\lambda_{k+1}$ can be uniquely paired with a local maximum curve $\lambda_m$ at $(t_\ell,h_\ell)$ for some $m > k+1$. Moreover, since $t_\ell<x_{k+1} \leq x_m$, we have $\eta_{m-1} = \lambda_m$ in a sufficiently small neighborhood of $t_\ell$. Therefore, $\eta_k$ is paired with the local maximum curve $\eta_{m-1}$ at $(t_\ell,h_\ell)$, in this case.
    
    \paragraph{Case 2} Suppose that $t_\ell > x_{k+1}$. In this case, we can find a small neighborhood of $t_\ell$ such that $\eta_k = \lambda_k$, which implies that $(t_\ell,h_\ell)$ is a local minimum of $\lambda_k$. Since $\lambda$ has non-negative degree, there exists a unique $m>k$ such that $(t_\ell,h_\ell)$ is a local maximum of $\lambda_m$. To show that $\lambda_m$ corresponds to some $\eta_j$, we first observe the following fact: $f_{(x_1 -y_1,x_1+y_1)}(t_\ell) \leq h_\ell$. Indeed, if $f_{(x_1 -y_1,x_1+y_1)}(t_\ell) > h_\ell$, it follows that
    \[
    \lambda_{k+1}(t_\ell) \geq f_{(x_1 -y_1,x_1+y_1)}(t_\ell)> h_\ell = \lambda_k(t_\ell),
    \]
    which is impossible since $\lambda_k \geq \lambda_{k+1}$ always. Therefore, $f_{(x_1 -y_1,x_1+y_1)}(t_\ell) \leq h_\ell$. This observation implies that $x_m \leq t_\ell$ (which can also be shown by contradiction similarly). Now, if $x_m < t_\ell$, then $\eta_m = \lambda_m$ sufficiently near $t_\ell$. Otherwise, if $x_m = t_\ell$, we have
    \begin{align*}
            \eta_{m-1}(t^{-}) &= \lambda_m(t^{-}) \text{ for } t^{-}<t_\ell \text{ sufficiently near } t_\ell \text{, and }\\
            \eta_m(t^{+}) &= \lambda_m(t^{+}) \text{ for } t^{+} \geq t_\ell \text{ sufficiently near } t_\ell.
    \end{align*}
    Since $\eta_{m-1} \geq \eta_m$, this forces $\eta_{m-1}(t^{-}) = \eta_m(t^{-})$ for $t^{-}<t_\ell$ near $t_\ell$, and thus $\eta_m(t) = \lambda_{m}(t)$ for all $t$ in a sufficiently small neighborhood of $t_\ell$. Hence, $\eta_k$ is paired with $\eta_m$ in this case.

    Having established that each $\eta_k$ pairs with some $\eta_j$, we proceed to show that this pairing is unique. Suppose that $\eta_k$ and $\eta_{k'}$, with $k \neq k'$, both attain a local minimum at $(t_\ell,h_\ell)$. Using arguments similar to those above, one can show that both $x_{k+1}$ and $x_{k'+1}$ must lie on the same side of $t_\ell$; that is, either $t_\ell < x_{k+1},x_{k'+1}$ or $t_\ell > x_{k+1},x_{k'+1}$. Assuming the former, it is easy to see that $\eta_k$ and $\eta_{k'}$ are paired with local maximum curves $\eta_m$ and $\eta_{m'}$ for some distinct $m\neq m'$. The latter case follows similarly. In conclusion, we have $\mathrm{deg}_\eta(t_\ell,h_\ell) \geq 0$, and thus $\eta$ is a landscape sequence.
    
    \smallskip
   \noindent\textbf{Item 2.}  We claim that $\eta$ admits strictly fewer local maxima than $\lambda$. By construction, every local maximum of $\eta$ is a local maximum of $\lambda$. Moreover, since $\lambda_1(x_1) = y_1 \neq \eta_k(x_1)$ for any $k$, the point $(x_1,y_1)$ is a local maximum of $\lambda$, but not of $\eta$. Therefore, $\eta$ must have strictly fewer local maxima (counting multiplicity) than $\lambda$.
   
\smallskip
    \noindent\textbf{Item 3.}  We aim to show that $\lambda = f_{(x_1-y_1,x_1 + y_1)}\oplus \eta $. For notational simplicity, let us write $f = f_{(x_1-y_1,x_1+y_1)}$ and $\tau = f \oplus \eta $. We claim that $\tau_k(t) = \lambda_k(t)$ for any $k\in \N,t\in \R$. Fix $t\in \R$. There are three cases to consider: either $t<x_1$, or $x_{n+1} \leq t$, or $x_{i} \leq t < x_{i+1}$ for some $1 \leq i \leq n$. In each case, the key is to determine whether $\tau_k(t)$ is realized by the first summand $f$ or the second $\eta$. We present only the last case, as the others follow similarly. So, assume that $x_i \leq t < x_{i+1}$ for some $1 \leq i \leq n$. Then:
    \begin{align*}
        f(t) &= \lambda_i(t), \\
        \eta_{k}(t) &= \lambda_k(t) \text{ for } 1\leq k \leq i-1, \text{ and } \\
        \eta_{k'}(t) &= \lambda_{k'+1}(t)\text{ for } i \leq k'.
    \end{align*}
    This gives the chain of inequalities:
    \begin{align*}
        \eta_1(t) = \lambda_1(t)  \geq \dots \geq \eta_{i-1}(t) = \lambda_{i-1}(t) \geq f(t) = \lambda_{i}(t) \geq \eta_i(t) =  \lambda_{i+1}(t) \geq \dots
    \end{align*}
    Hence, $\tau_k(t) = \lambda_k(t)$ as desired.
\end{proof}

With Lemma \ref{lem: Landscape Decomposition Theorem Lemma} established, we can now prove the Landscape Decomposition Theorem \ref{thm: Landscape Decomposition Theorem}.
\begin{proof}[Proof of Theorem \ref{thm: Landscape Decomposition Theorem}]
    The uniqueness claim is straightforward. Indeed, given a  landscape sequence of the form $f_{(b_1,d_1)} \oplus f_{(b_2,d_2)} \oplus \cdots \oplus f_{(b_N,d_N)}$, observe that the points 
    \[
    \left(\frac{b_i+d_i}{2}, \frac{d_i-b_i}{2} \right)
    \]
    are exactly the local maxima with strictly positive degree. As the multiset of such points is an intrinsic property of a landscape sequence, any two decompositions can only differ by reordering summands.
    
    It remains to prove the existence of a direct sum decomposition of a landscape sequence $\lambda$. Let $M$ denote the number of local maxima of $\lambda$. We argue by induction on $M$.
    
    To establish the base case, assume $M = 1$. Let $(x_1,y_1)$ be the unique local maximum of $\lambda$. It must correspond solely to the local maximum of $\lambda_1$. Moreover, we must have $\lambda_1(x_1-y_1) = 0 = \lambda_1(x_1+y_1)$; otherwise, $\lambda_1$ would attain a local minimum in the interval $[x_1-y_1,x_1+y_1]$ which would induce an additional local maximum curve in $\lambda$. Lastly, since there is only one local maximum, we must have $\lambda_k = 0$ for all $k \geq 2$. This shows that $ \lambda = f_{(x_1-y_1,x_1+y_1)}$, proving the base case.
    
    Now assume that every landscape sequence with at most $M-1$ local maxima admits a decomposition as in the statement of this theorem, and assume that $\lambda$ has $M$ local maxima. Let $(x_1,y_1)$ be the left-most point where $\lambda_1$ attains a local maximum.
    By Lemma \ref{lem: Landscape Decomposition Theorem Lemma}, we have a landscape sequence $\eta$ that admits strictly fewer local maxima than $\lambda$ and satisfies $\lambda = f_{(x_1-y_1,x_1 + y_1)} \oplus \eta$. Therefore, by the induction hypothesis, $\eta$ admits a decomposition, and so does $\lambda$. This concludes the proof of existence.
\end{proof}

\begin{figure}
    \centering
    \includegraphics[width=1\linewidth]{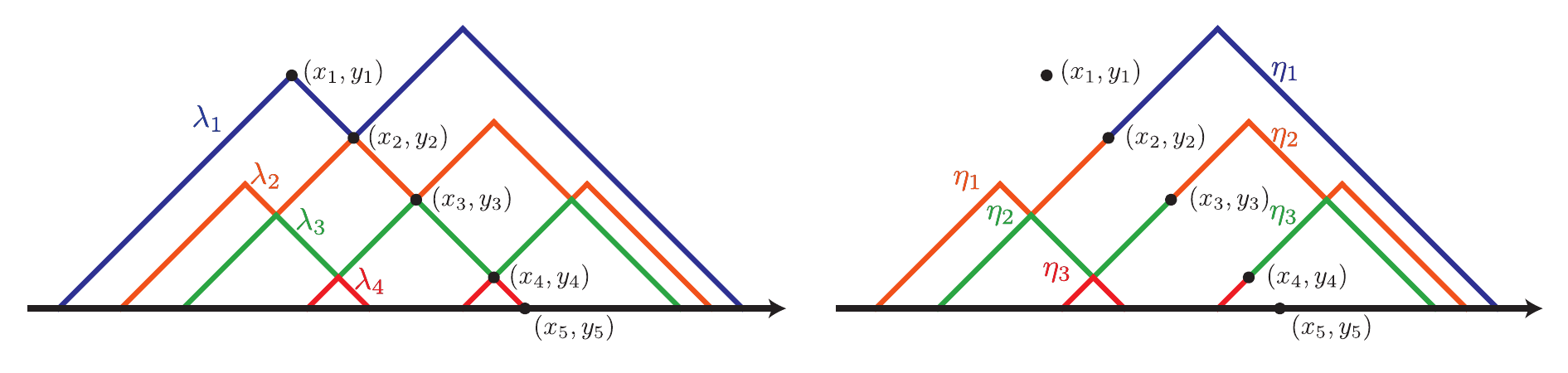}
    \caption{Example demonstrating the effect of removing a tent function from a landscape sequence, as described in Lemma \ref{lem: Landscape Decomposition Theorem Lemma}. \textbf{Left}: Original landscape sequence $\lambda$. The points $(x_i,y_i)$ as in the statement of the lemma are annotated. \textbf{Right}: The resulting landscape sequence $\eta$ obtained from $\lambda$ by removing the tent function $f_{(x_1-y_1,x_1 + y_1)}$. The colors here correspond to those of the original landscape $\lambda$, changing at the points $x_i$, indicating a transition from one landscape curve to the other, in line with the definition of $\eta$, as described in the lemma.
    }
\label{fig:removing_tent_from_landscape_sequence}
\end{figure}

Finally, we leverage the direct sum decomposition to give a proof of Theorem \ref{thm: Every landscape sequence is a persistence landscape of some persistence diagram}.

\begin{proof}[Proof of Theorem \ref{thm: Every landscape sequence is a persistence landscape of some persistence diagram}]
    Let $Y = \{ (b_i,d_i)\}_{i=1}^N$ be a persistence diagram. It is easy (by definition of $\lambda^Y$) to see that $\lambda^Y = f_{(b_1,d_1)} \oplus f_{(b_2,d_2)} \oplus \dots \oplus f_{(b_N,d_N)}$. Therefore, $\lambda^Y$ is a landscape sequence by Proposition \ref{prop: Degree nonzero prop}. 
    
    Conversely, let $\lambda$ be a landscape sequence. By Theorem \ref{thm: Landscape Decomposition Theorem},
    there exists a unique multiset of tent functions $\{f_{(b_1,d_1)}, \dots, f_{(b_N,d_N)}\}$ such that 
    \begin{align*}
        \lambda = f_{(b_1,d_1)} \oplus f_{(b_2,d_2)} \oplus \dots \oplus f_{(b_N,d_N)}.
    \end{align*}
    Now, define the persistence diagram $Y = \{(b_i, d_i)\}_{i=1}^N$ (this is well-defined, by the uniqueness part of the decomposition theorem). By construction, $\lambda^Y$ has the same decomposition as $\lambda$. Thus, $\lambda = \lambda^Y$.
\end{proof}

\subsection{Equivalence of Erosion and Landscape Distances}

In this section, we prove our main theorem, which equates erosion distance and the landscape distance \eqref{eqn:landscape_distance_background}.

\begin{thm}[Equivalence of Erosion and Landscape Distances] \label{thm: d_E = Landscape Distance}
    The landscape distance between any two persistence landscapes is equal to the erosion distance between their corresponding persistence diagrams. That is, $\| \lambda^Y - \lambda^{Y'} \|_{\infty} = d_{E}(Y,Y')$ for any $Y,Y' \in \mathsf{PDgm}$.
\end{thm}

The proof follows by realizing the landscape distance as an interleaving distance with respect to a coflow and showing that the resulting structure is equivalent to the one established for persistence diagrams in Section \ref{subsec:erosion_distance}.

\begin{remark}\label{rmk:KMS_result}
    In \cite[Theorem 4.3]{kim2024interleaving}, the authors show that the erosion distance can be realized as a Hausdorff distance in the space $\R_<^2$ of intervals, endowed with $\ell^\infty$ distance. After unpacking definitions, one can show that the latter quantity is equal to landscape distance, thereby providing an alternate proof strategy for Theorem \ref{thm: d_E = Landscape Distance}. Our perspective provides a more elementary approach (indeed, \cite[Theorem 4.3]{kim2024interleaving} is an example of their rather general category-theoretic framework), while leveraging the structural results for persistence landscapes derived above. 
\end{remark}

\subsubsection{Equipping Persistence Landscapes with a Coflow}\label{subsec:coflow_on_landscapes}

 We now define the poset category of landscape sequences, and equip it with a coflow.  

\begin{defn}
    By $ \mathsf{Land} = (\mathsf{Land}, \leq )$, we denote the \textbf{poset of landscape sequences} $\lambda$ where the poset relation $\leq$ is given by $\lambda \leq \lambda' \iff \lambda_{k}(t) \leq \lambda'_{k}(t)$ for all $k\in\N,t\in \R$.
\end{defn}

Let $\varepsilon\geq 0$ and $\lambda = (\lambda_1,\lambda_2,\dots)$. Define an operator $T_{\varepsilon}$ on landscape curves by 
\begin{align*}
    T_{\varepsilon}(\lambda_k): & \R \to \R 
     && t \mapsto T_{\varepsilon}(\lambda_k)(t) =
    \begin{cases} 
        \lambda_k(t) - \varepsilon, & \text{if } \lambda_k(t) > \varepsilon \\
        0, & \text{otherwise.} 
    \end{cases}
\end{align*}

That is, \( T_{\varepsilon}(\lambda_k) \) is obtained from \( \lambda_k \) by shifting the graph of \( \lambda_k \) downward by \( \varepsilon \) units, and truncating as necessary. Observe that if $\lambda_k' \leq \lambda_k$, then $T_\varepsilon(\lambda_k') \leq T_\varepsilon(\lambda_k)$. This operation therefore gives rise to the following well-defined functor.

\begin{defn}
    For each $\varepsilon \geq 0$, define the \textbf{landscape $\varepsilon$-flow functor} $\Lambda_{\varepsilon}:\mathsf{Land} \to \mathsf{Land}$, by $\Lambda_{\varepsilon}\lambda = (T_{\varepsilon}(\lambda_1),T_{\varepsilon}(\lambda_2),\dots)$.
\end{defn}

The proof of the following is straightforward.

\begin{prop}
    The family $\Lambda = \{ \Lambda_{\varepsilon} \}_{\varepsilon \geq 0}$ is a coflow.
\end{prop}

The interleaving distance induced by $\Lambda$ is familiar: 

\begin{prop}\label{prop:coflow_is_landscape}
    The interleaving distance induced by the coflow $\Lambda$ is equal to the landscape distance. That is,
    \[
    d^{\Lambda}_{\mathsf{Land}}(\lambda,\lambda') = \| \lambda - \lambda'\|_{\infty}
    \]
    for all $\lambda,\lambda' \in \mathsf{Land}$. 
\end{prop}

\begin{proof}
    Our construction of the interleaving distance $d_{\mathsf{Land}}^{\Lambda}$ as a distance between functions $\mathbb{N} \times \R \to \R$ agrees with the construction used to prove \cite[Theorem 3.9]{de2018theory}, which shows that $L^\infty$ distance on $\R^n$ can be realized as an interleaving distance. Indeed, small modifications to that proof yield the proof in the setting of this proposition.
\end{proof}

\subsubsection{Coflow Equivalence of Persistence Landscapes and Persistence Diagrams}

Let us denote the \textbf{landscape map} $Y \mapsto \lambda^Y$ by $\mathcal{L}:\mathsf{PDgm} \to \mathsf{Land}$. Theorem \ref{thm: Every landscape sequence is a persistence landscape of some persistence diagram} shows that the landscape map is a bijection; let $\mathcal{M}:\mathsf{Land} \to \mathsf{PDgm}$ denote its inverse. In this section, we refine this result and show that $\mathcal{L}$ is a coflow equivalence between the posets $(\mathsf{PDgm}, \leq)$ and $(\mathsf{Land}, \leq)$.

\begin{prop}\label{prop: all morphisms preserved}
    Let $Y,Y' \in \mathsf{PDgm}$. Then
    \[
    Y \leq Y' \iff \lambda^{Y} \leq \lambda^{Y'}
    \]
    That is, morphisms are preserved among the posets $\mathsf{PDgm}$ and $\mathsf{Land}$, so that $\mathcal{L}$ and $\mathcal{M}$ are functors.
\end{prop}

\begin{proof}
    First, suppose that $Y \leq Y'$. Let $k \in \mathbb{N}$ and $t \in \R$, and assume that there exists $h \geq 0$ such that $\mathrm{rank}[Y](t-h,t+h) \geq k$. Then
    \[
    k \leq \mathrm{rank}[Y](t-h,t+h) \leq \mathrm{rank}[Y'](t-h,t+h).
    \]
    Since this holds for all $h$ satisfying the condition, this shows that $\lambda_k^{Y} \leq \lambda_k^{Y'}$; if no such $h$ exists, then the desired conclusion is obvious, as $\lambda_k^{Y} = 0$. As $k$ was arbitrary, this shows that $\lambda^{Y} \leq \lambda^{Y'}$. 

    Conversely, suppose that $\lambda^{Y} \leq \lambda^{Y'}$. Let $I = (b,d) \in \R_\leq^2$ and suppose that $\mathrm{rank}[Y](I) = k$. If $k = 0$, the desired conclusion that $\mathrm{rank}[Y](I) \leq \mathrm{rank}[Y'](I)$ is obvious, so let us suppose that $k \geq 1$. Then we have 
    \[
    k = \mathrm{rank}[Y](b,d) = \mathrm{rank}[Y]\left(\frac{d+b}{2} - \frac{d-b}{2}, \frac{d+b}{2} + \frac{d-b}{2} \right). 
    \]
    This implies
    \[
    \frac{d-b}{2} \leq \lambda_k^{Y}\left(\frac{d+b}{2}\right) \leq \lambda_k^{Y'}\left(\frac{d+b}{2}\right),
    \]
    which, in turn, implies that $\mathrm{rank}[Y'](I) \geq k$. Since $I$ was arbitrary, this implies that $Y \leq Y'$. 
\end{proof}

\begin{prop} \label{prop: truncated landscape & its diagram}
    The landscape functor $\mathcal{L}:\mathsf{PDgm}\to \mathsf{Land}$ is coflow-equivariant; that is,  $\lambda^{\nabla_{\varepsilon}Y} = \Lambda_{\varepsilon}\lambda^{Y}$ for every $Y \in \mathsf{PDgm}$ and $\varepsilon \geq 0$. 
\end{prop}

\begin{proof}
    Let $k\in \N$, $t\in \R$ and $\varepsilon \geq 0$ be arbitrary. We will show that $\lambda_k^{\nabla_{\varepsilon}Y}(t) = T_{\varepsilon}(\lambda_k^{Y})(t)$. First, assume that $\lambda^{Y}_k(t) - \varepsilon > 0$. Then
    \begin{align*}
        T_{\varepsilon}(\lambda_k^{Y})(t) =& \lambda_k^{Y}(t) - \varepsilon \\
        =& \sup\{ h \geq 0 \, \mid \, \mathrm{rank}[Y] (t-h,t+h)\geq k\} - \varepsilon \\
        =& \sup\{ h - \varepsilon \geq 0 \, \mid \, \mathrm{rank}[Y] (t-h,t+h)\geq k\}\\
        =& \sup\{ h' \geq 0 \, \mid \, \mathrm{rank}[Y] (t-h'-\varepsilon,t+h'+\varepsilon)\geq k\}\\
        =& \lambda_{k}^{\nabla_{\varepsilon}Y}(t)
    \end{align*}
    On the other hand, suppose that $\lambda^{Y}_k(t) \leq \varepsilon $. Note that this implies that for any $h \geq 0$ such that $\mathrm{rank}[Y](t-h,t+h) \geq k$, we have $h \leq \varepsilon$. Now, let $h\geq 0$ be such that $\mathrm{rank}[\nabla_{\varepsilon}Y](t-h,t+h) \geq k$. We will show that $h$ is forced to be $0$. Indeed, by Lemma \ref{prop 5.1}, we have 
    \[
    \mathrm{rank}[Y](t-(h+\varepsilon),t+ (h+ \varepsilon))\geq k
    \]
    but as noted this means $h + \varepsilon \leq \varepsilon$. So, $h = 0$. That is, $\lambda^{\nabla_{\varepsilon}Y}_k(t) = 0 = T_{\varepsilon}(\lambda_k^Y)(t)$. 
\end{proof}

The proof of our main result, Theorem \ref{thm: d_E = Landscape Distance}, now follows easily.

\begin{proof}[Proof of Theorem \ref{thm: d_E = Landscape Distance}]
    Propositions \ref{prop: all morphisms preserved} and \ref{prop: truncated landscape & its diagram} show that the landscape map $\mathcal{L}:\mathsf{PDgm} \to \mathsf{Land}$ is a coflow equivariant equivalence (see Definition \ref{def:coflow_equivariant}). By Proposition \ref{prop: Coflow induced interleaving distances between coflow equivalent poset categories are preserved}, it follows that the interleaving distance induced by the family $\Lambda$ is equal to the erosion distance; that is, $d_{\mathsf{Land}}^{\Lambda}(\lambda^{Y},\lambda^{Y'}) = d_{E}(Y,Y')$. By Proposition \ref{prop:coflow_is_landscape}, $d_{\mathsf{Land}}^{\Lambda}(\lambda^{Y},\lambda^{Y'}) = \|\lambda^Y - \lambda^{Y'}\|_\infty$, so this completes the proof.
\end{proof}

\section{Metric Properties of the Erosion Distance} \label{section: Topological Results}

This section collects some topological results regarding the erosion distance, many of which are consequences of our main equivalence result, Theorem \ref{thm: d_E = Landscape Distance}. 

Throughout this section, we mildly abuse terminology and refer to the metric on $\mathsf{PDgm}$ which is induced by pulling back the sup norm distance under the landscape map as the \textbf{landscape distance}. We use the notation
\begin{equation}\label{eqn:landscape_distance_diagrams}
    d_L(Y,Y') \coloneqq \|\lambda^Y - \lambda^{Y'}\|_\infty.
\end{equation}
This convention is introduced to smoothen certain statements of our results.

Some of our results concern coarse embeddings of metric spaces~\cite{gromov1993asymptotic,roe2003lectures}, so we recall the basic definition here. Let $(X,d_X)$ and $(Y,d_Y)$ be metric spaces and let $f:X \to Y$ be a map. We say that $f$ is a \textbf{coarse embedding} if there exist non-decreasing maps $\rho^-,\rho^+:[0,\infty) \to [0,\infty)$, with $\lim_{t \to \infty} \rho^-(t) = \infty$, such that 
\[
\rho^-(d_X(x,x')) \leq d_Y(f(x),f(x')) \leq \rho^+(d_X(x,x'))
\]
for all $x,x' \in X$. We say that two metrics $d$ and $d'$ on the same set $X$ are \textbf{coarsely equivalent} if the identity map coarsely embeds $d$ into $d'$ and vice versa. 

\subsection{Topological Equivalence of Persistence Diagrams Under Erosion and Bottleneck Distances} \label{Subsection: Topological Equivalence}

In this subsection, we prove that the  bottleneck distance $d_B$ (see Section \ref{subsec:bottleneck_distance}) induces the same topology on persistence diagrams as the erosion distance, $d_E$. First, note that the bottleneck distance serves as an upper bound for the erosion distance, i.e., one has $d_E(Y, Y') \leq d_B(Y, Y')$ for any pair of persistence diagrams $Y$ and $Y'$ (see \cite[Lemma 6.4]{xian2022capturing}). This implies that a bottleneck-distance open ball, $B_{d_B}(Y, r)$, lies inside an erosion-distance open ball, $B_{d_E}(Y, r)$. In other words, every $d_E$-open set $U$ is a $d_B$-open set. However, the bottleneck and erosion distances are not bi-Lipschitz equivalent. In fact, we can say something stronger:

\begin{prop}\label{prop:bottleck_and_erosion_not_coarsely_equivalent}
    The bottleneck and erosion distances are not coarsely equivalent.
\end{prop}

\begin{proof}
    We will show that there is no function $\rho^-:[0,\infty) \to [0,\infty)$ with $\displaystyle \lim_{t \to \infty} \rho^- (t) = \infty$ such that 
    \[
    \rho^-(d_B(Y,Y')) \leq d_E(Y,Y')
    \]
    for all $Y,Y' \in \mathsf{PDgm}$.
    In the proof of \cite[Proposition 21]{bubenik2020persistence}, Bubenik constructs, for any positive integer $n$, a pair of diagrams $Y_n$ and $Y_n'$ such that 
    \[
    \|\lambda^{Y_n} - \lambda^{Y_n'}\|_\infty = 1 \quad \mbox{and} \quad d_B(Y_n,Y_n') = 2n + 1.
    \]
    By Theorem \ref{thm: d_E = Landscape Distance}, the first equality implies $d_E(Y_n,Y_n') = 1$. Now let $\rho^-:[0,\infty) \to [0,\infty)$ be any function with $\lim_{t \to \infty} \rho^-(t) = \infty$. Choose $n$ such that $\rho^-(2n+1) > 1$. Then 
    \[
    d_E(Y_n,Y_n') = 1 < \rho^-(d_B(Y_n,Y_n')).
    \]
    This proves our initial non-existence claim.
\end{proof}

Despite the fact that the bottleneck and erosion distances are not coarsely equivalent, we now prove that they are \emph{locally} equivalent.

\begin{thm}[Local Equivalence of Bottleneck and Erosion Distances] \label{thm: bottleneck is erosion}
    The bottleneck and erosion distances generate the same topology on $\mathsf{PDgm}$. In particular, let $Y = \{(b_i,d_i)\}_{i=1}^N$ be a nonempty persistence diagram and set 
    \begin{align*}
        r_Y \coloneqq \frac{1}{8}\min\left\{
   |u-v| \mid u,v \in  \{b_i\}_{i=1}^N \cup \{d_i\}_{i=1}^N \mbox{ and } u \neq v
    \right\}.
    \end{align*}
    If $Y' \in \mathsf{PDgm}$ satisfies $d_E(Y,Y') < r_Y$ then 
    \[
    d_E(Y,Y') \leq d_B(Y,Y') \leq 3 \cdot d_E(Y,Y').
    \] 
    If, in addition, $Y$ and $Y'$ have the same cardinality (counting multiplicity) then
    \[
    d_E(Y,Y') = d_B(Y,Y').
    \]
    Finally, if $Y$ is the empty diagram, then for any diagram $Y'$,
    \[
    d_E(Y,Y') = d_B(Y,Y') = \left\{
    \begin{array}{cl}
    \max\left\{\frac{d_j' - b_j'}{2} \mid (b_j',d_j') \in Y'\right\} & \mbox{if $Y'$ is nonempty,} \\
    0 & \mbox{otherwise.}\end{array}\right.
    \]
\end{thm}

Combining the above theorem with Theorem \ref{thm: d_E = Landscape Distance}, we immediately derive:

\begin{cor}
    The landscape distance $d_L$ and the bottleneck distance $d_B$ induce the same topology on $\mathsf{PDgm}$. 
\end{cor}

To prove Theorem \ref{thm: bottleneck is erosion}, we rely on the following technical lemma.

\begin{lem}\label{lem:open_ball_lemma}
    Let $Y = \{(b_i,d_i)\}_{i=1}^N \in \mathsf{PDgm}$ be nonempty and set
\[
\bar r_Y \coloneqq \frac12\min\left(
\{|b_i-b_j| \mid  b_i\neq b_j\}
\cup
\{|d_i-d_j| \mid \ d_i\neq d_j\}
\cup
\left\{\frac{d_i-b_i}{2} \right\}
\right).
\]
    Suppose that  $Y' \in \mathsf{PDgm}$ satisfies $d_E(Y,Y') < r < \bar{r}_Y$. Then for each $(b,d) \in Y$ with multiplicity $m$, there are exactly $m$ birth-death pairs in $Y'$ (counted with multiplicity) that lie in the $r$-unit box 
        \[
        Q \coloneqq \{(x,y)\in \rr \, \mid \, d_\infty((b,d),(x,y))<r \}
        \]
        centered at $(b,d)$.
\end{lem}

\begin{proof}
Suppose that $d_E(Y,Y')<r$. Then there exists $\varepsilon \geq 0$ with $d_E(Y,Y') \leq \varepsilon < r$ such that the two rank conditions in Equation \ref{eqn:erosion_distance_initial} are satisfied. Let $(b,d) \in Y$ be an arbitrary pair with multiplicity $m \geq 1$. We consider the four distinct corners of the open box $Q$; that is
    \begin{itemize}
        \item top-left: $I_1 = (b-r,d+r)$,
        \item top-right: $I_2 = (b+r,d+r)$,
        \item bottom-left: $I_3 = (b-r,d-r)$,
        \item bottom-right: $I_4 = (b+r,d-r)$.
    \end{itemize}
    These give rise to four disjoint regions below:
    \begin{itemize}
        \item Region $1$: $\{ (x,y) \in \rr \, \mid \, (x,y)\text{ properly contains } I_1 \} $
        \item Region $2$: $\{ (x,y) \in \rr \, \mid \, (x,y)\text{ properly contains } I_3, \text{ but not $I_1$} \}   $
        \item Region $3$: $\{ (x,y) \in \rr \mid (x,y)$ properly contains $I_2$ but not $I_1 \}$ 
        \item Region $4$: the half-open box $\widetilde Q_r=(b-r,b+r]\times(d-r,d+r]$
    \end{itemize}
    Assume that $Y$ has $n_1 \geq 0$ points in Region 1, $n_2 \geq 0$ points in Region 2 and $n_3 \geq 0$ points in Region 3. Since $(b,d)$ has multiplicity $m$, and it properly contains $I_4$, we have $\mathrm{rank}[Y](I_4) = m +n_1 + n_2 + n_3$. Define $\mathrm{Shrink}_{\varepsilon}(x,y) \coloneqq (x+\varepsilon,y-\varepsilon)$. Since $\varepsilon < r < \bar{r}_Y$,  we have that, for each $I \in \{I_1,I_2,I_3,I_4\}$,
    \[
    \mathrm{rank}[\nabla_\varepsilon Y](I) = \mathrm{rank}[Y](I) = \mathrm{rank}[Y](\mathrm{Shrink}_\varepsilon I) 
    \]
    and moreover $\mathrm{Grow}_{\varepsilon}\circ \mathrm{Shrink}_{\varepsilon}(I) = I$. Indeed, let us explain this in the case $I=I_1$, with the other arguments for the other points being similar. We have, by definition, that 
    \[
     \mathrm{rank}[\nabla_\varepsilon Y](I_1) \leq \mathrm{rank}[Y](I_1) \leq \mathrm{rank}[Y](\mathrm{Shrink}_\varepsilon I_1).
    \]
    On the other hand, $\mathrm{rank}[\nabla_\varepsilon Y](I_1)$ is the number of points $(b_i,d_i) \in Y$ such that $(b_i+\varepsilon,d_i-\varepsilon)$ properly contains the interval $I_1 = (b-r,d+r)$, and $\mathrm{rank}[Y](I_1)$ is the number of points  $(b_i,d_i) \in Y$ such that $(b_i,d_i)$ properly contains $I_1$. Suppose that $(b_i , d_i )$ properly contains $I_1$; that is, $b_i  \leq b-r  < d+r < d_i $. We wish to show that $b_i + \varepsilon \leq b-r  < d+r < d_i - \varepsilon$, whence it follows that $  \mathrm{rank}[\nabla_\varepsilon Y](I_1) \geq \mathrm{rank}[Y](I_1) $. To obtain a contradiction, suppose, without loss of generality that $b_i + \varepsilon > b-r$. Then $b-b_i<r+\varepsilon<2\bar{r}_Y$. On the other hand, $b_i\leq b-r<b$, so $b_i$ and $b$ are distinct birth coordinates of $Y$. By the definition of $\bar{r}_Y$, this implies $b-b_i\geq 2\bar{r}_Y$, a contradiction. 
    Similarly, one can show that $d+r < d_i - \varepsilon$. The fact that $\mathrm{rank}[Y](I_1) \geq \mathrm{rank}[Y](\mathrm{Shrink}_\varepsilon I_1)$ follows by the same type of argument. Moreover, by the fact that $\varepsilon < \bar{r}_Y$, it follows immediately that  $\mathrm{Grow}_{\varepsilon}\circ \mathrm{Shrink}_{\varepsilon}(I_1) = I_1$.
    
    Since $\varepsilon$ satisfies the rank conditions, we have
    \[
    \mathrm{rank}[\nabla_\varepsilon Y](I) \leq \mathrm{rank}[Y'](I)\leq \mathrm{rank}[Y](\mathrm{Shrink}_{\varepsilon}(I))
    \]
    which in turn yields the equality
    \[
     \mathrm{rank}[Y'](I) = \mathrm{rank}[Y](I)
    \]
    for all $I\in \{I_1,I_2,I_3,I_4\}$. Explicitly, this means:
    \begin{itemize}
        \item $\mathrm{rank}[Y'](I_1) = \mathrm{rank}[Y](I_1) = n_1$; that is, $Y'$ must admit exactly $n_1$ pairs that properly contain $I_1$, all of which lie in Region 1 by definition.
        \item $\mathrm{rank}[Y'](I_2) = \mathrm{rank}[Y](I_2) = n_1 + n_3$; that is, $Y'$ must admit exactly $n_1 + n_3$ pairs that properly contain $I_2$. This implies that there are exactly $n_3$ pairs in $Y'$ that lie in Region 3.
        \item $\mathrm{rank}[Y'](I_3) = \mathrm{rank}[Y](I_3) = n_1 + n_2$; that is, $Y'$ must admit exactly $n_1 + n_2$ pairs that properly contain $I_3$. This implies that there are exactly $n_2$ pairs in $Y'$ that lie in Region 2.
        \item $\mathrm{rank}[Y'](I_4) = \mathrm{rank}[Y](I_4) = m + n_1 + n_2 + n_3$; that is, $Y'$ must admit exactly $m + n_1 + n_2 + n_3$ pairs that properly contain $I_4$.
    \end{itemize}
    Thus $Y'$ contains exactly $m$ points in the half-open box $\widetilde Q_r=(b-r,b+r]\times(d-r,d+r]$ counting multiplicity. To show that these points lie in the open box $Q$, choose $s$ with $\varepsilon<s<r$. Observe that the preceding argument applies to $s$ as well. Hence $Y'$ also contains exactly $m$ points in $\widetilde Q_s=(b-s,b+s]\times(d-s,d+s]$. Because $\widetilde Q_s\subset Q\subset\widetilde Q_r$, we obtain
    \[
    m=|Y'\cap\widetilde Q_s|
    \leq |Y'\cap Q|
    \leq |Y'\cap\widetilde Q_r|=m,
    \]
    where all counts include multiplicity. Therefore $Y'$ contains exactly $m$ points in $Q$, as claimed.
    
\end{proof}

\begin{terminology} We introduce some terminology related to the definition of bottleneck distance, given in Section \ref{subsec:bottleneck_distance}.
\begin{itemize}
    \item We will refer to a triplet $(S, S', \phi)$, with $S \subset Y$, $S' \subset Y'$, and $\phi:S \to S'$ a bijection, as a \textbf{partial matching} between $Y$ and $Y'$.
    \item Given a partial matching one can interpret the quantity
    \[
    \max\left\{ \max_{(b, d) \in S} d_\infty((b, d), \phi(b, d)), \; \max_{(b, d) \notin S} \frac{(d - b)}{2}, \; \max_{(b', d') \notin S'} \frac{(d' - b')}{2} \right\}
    \]
    as the \textbf{cost} of $(S,S',\phi)$ and denote it as $\mathrm{cost}(\phi)$. 
    \item We say that a partial matching $(S,S',\phi)$ is \textbf{optimal} if it achieves the minimum cost.
    \item By $(b,d) \overset{\phi}{\longleftrightarrow} (b',d')$, we will denote that $\phi(b,d) = (b',d')$ and say that $(b,d)$ and $(b',d')$ are \textbf{matched} under $\phi$. Similarly, by $(b,d) \overset{\phi}{\longleftrightarrow} \Delta$, we will denote that $(b,d)$ is \textbf{matched with the diagonal} under $\phi$, i.e., $(b,d) \notin S$.
\end{itemize}
\end{terminology}

We now proceed with the proof of the main result of this subsection.

\begin{proof}[Proof of Theorem \ref{thm: bottleneck is erosion}] 
	The statements regarding empty diagrams are obvious from the definitions, so let us assume diagrams are nonempty for the remainder of the proof. Let $Y' \in \mathsf{PDgm}$ such that $d_E(Y,Y') < r < r_Y < \bar{r}_Y$ (with $\bar{r}_Y$ as in Lemma \ref{lem:open_ball_lemma}). The inequality $d_E(Y,Y') \leq d_B(Y,Y')$ holds, in general. Referring to the boxes used in Lemma \ref{lem:open_ball_lemma}, let us denote the $r$-box around $(b_i,d_i) \in Y$ as $Q_i$ for the rest of the proof.  
	
	First, suppose that $Y'$ has the same cardinality as $Y$, counting multiplicity. By Lemma \ref{lem:open_ball_lemma}, there exists a bijection between the points of $Y$ and $Y'$, which is, in particular a partial matching; indeed, for each $(b_i,d_i) \in Y$ of multiplicity $m$, there exist exactly $m$ points from $Y'$ in $Q_i$, and the boxes $Q_i$ and $Q_j$ are disjoint if $(b_i,d_i) \neq (b_j,d_j)$, since $r < r_Y$. The cost of this partial matching is strictly bounded above by $r$. By suboptimality of this partial matching, and since this construction works for all $r$ satisfying $d_E(Y,Y') < r < r_Y$, it follows that $d_B(Y,Y') \leq d_E(Y,Y')$. This proves that $d_E(Y,Y') = d_B(Y,Y')$ when $Y$ and $Y'$ have the same cardinality.
	
	Once again, by Lemma \ref{lem:open_ball_lemma}, the cardinality of $Y'$ must be at least that of $Y$. It then remains to consider the case that $Y'$ has strictly larger cardinality than $Y$. In this case, we still obtain a partial matching by starting with the matching induced by Lemma \ref{lem:open_ball_lemma} (match points of $Y'$ contained in $Q_i$ to copies of $(b_i,d_i)$), then matching remaining points of $Y'$ to the diagonal; we denote this partial matching as $\phi$. Let $(b',d') \in Y'$ be matched to the diagonal under this partial matching. We claim that $d'-b' < 6r$, in which case the overall cost of the partial matching is strictly bounded above by $3r$, so that the same argument given in the previous paragraph yields the desired Lipschitz bound $d_B(Y,Y') \leq 3 \cdot d_E(Y,Y')$. 
	
	To complete the proof, we argue by contradiction: suppose that $d' - b' \geq 6r$. Choose $\varepsilon\geq 0$ with
$d_E(Y,Y')\leq\varepsilon<r$ such that the two rank conditions
in \eqref{eqn:erosion_distance_initial} are satisfied. Let
\[
S=\{b_i\}_{i=1}^N\cup\{d_i\}_{i=1}^N
\]
be the set of distinct endpoint values of $Y$. Recall that
\[
r_Y=\frac18\min\{|c-c'|:c,c'\in S,\ c\neq c'\}.
\]
We claim that there exists $t$ such that
\[
b'+\varepsilon<t<d'-\varepsilon
\qquad\text{and}\qquad
|t-c|>2r \quad\text{for every }c\in S.
\]
To see this, let $J=(b'+\varepsilon,d'-\varepsilon)$. Since
$d'-b'\geq 6r$ and $\varepsilon<r$, we have
\[
|J|=d'-b'-2\varepsilon>4r.
\]
For each $c\in S$, the interval $[c-2r,c+2r]$ has length $4r$.
Moreover, since $r<r_Y$, distinct endpoint values are more than
$8r$ apart, so these intervals are separated by gaps of length
greater than $4r$. The interval $J$ is too long to be contained in any one of these
intervals. Nor can it be contained in their union: in that case, $J$ would also contain points in a gap between them, which lie outside the union. Consequently, there exists
\[
t\in J\setminus\bigcup_{c\in S}[c-2r,c+2r],
\]
and this $t$ satisfies both required conditions.


Now consider the degenerate interval $I = (t,t)$. Suppose that $(b_i,d_i) \in Y$ contains $I$, and that $(b_i,d_i) \xleftrightarrow{\phi} (b_j',d_j')$ for some $(b_j',d_j') \in Y'$. We claim that $\mathrm{Shrink}_\varepsilon(b_j',d_j')$ also contains $I$. Indeed, by our choice of $t$, we have $b_i + 2r < t < d_i - 2r$ and the definition of $\phi$ tells us that $(b_j',d_j') \in Q_i$, whence 
\[
b_j' + \varepsilon = b_j' - b_i + b_i + \varepsilon < b_i + 2r < t \quad \mbox{and} \quad d_j' - \varepsilon = d_j' - d_i + d_i - \varepsilon > d_i - 2r > t.
\]

Since $(b',d')$ is, by assumption, matched to the diagonal, the above shows that there are strictly more points of the form $\mathrm{Shrink}_\varepsilon(b_j',d_j')$, with $(b_j',d_j') \in Y'$, which contain $I$ than there are points $(b_i,d_i) \in Y$ containing $I$. In other words, 
\[
\mathrm{rank}[\nabla_\varepsilon Y'](I) > \mathrm{rank}[Y](I),
\]
violating \eqref{eqn:erosion_distance_initial}. We have therefore arrived at a contradiction, and it must be that all unmatched points of $Y'$ have length strictly less than $6r$. Hence the cost of our partial matching $\phi$ is strictly less than $3r$. Since this holds for all $r$ with $d_E(Y,Y') < r < r_Y$, it follows that 
\[
d_B(Y,Y')\leq 3\cdot d_E(Y,Y').
\]
This completes the proof.
\end{proof}

\subsection{Erosion Distance is not a Length Metric}

As a corollary to Theorem \ref{thm: bottleneck is erosion}, we prove that the space of persistence diagrams with the erosion distance, $(\mathsf{PDgm},d_E)$, is not a length space, and, in fact, its intrinsic metric is the bottleneck distance. The following recalls some standard definitions from metric geometry (see, e.g., \cite{burago2001course,bridson2013metric}).

\begin{defn}
    Let $(X,d)$ be a metric space. Denote by $\mathrm{len}_d(\alpha)$ the \textbf{length} of a continuous path $\alpha:[0,1] \to X$. That is, 
    \[
        \mathrm{len}_d(\alpha) = \sup\bigg\{ \sum d\left(\alpha(t_i),\alpha(t_{i+1})\right)\, \mid \, 0 = t_0<t_1<\cdots<t_{n+1} = 1 \bigg\}.
    \]
    If this supremum is finite, $\alpha$ is called \textbf{rectifiable}.
    The \textbf{intrinsic metric} associated to $d$ is defined by
    \begin{equation}\label{eqn:intrinsic_metric}
    \hat{d}(x,x') \coloneqq \inf\{\mathrm{len}_d(\alpha) \mid \alpha(0) = x, \; \alpha(1) = x'\},
    \end{equation}
    where the infimum is over all rectifiable paths into $X$ satisfying the given boundary conditions. If $\hat{d}= d$, then $(X,d)$ is called a \textbf{length space}. If $(X,d)$ is a length space such that the infimum in \eqref{eqn:intrinsic_metric} is always realized, then $(X,d)$ is called a \textbf{geodesic space}. In this case, an infimizing path $\alpha$ is called a \textbf{geodesic}, and it satisfies, after constant speed reparameterization,
    \begin{align*}
       d(\alpha(s),\alpha(t)) = |s-t| d(x,x')
    \end{align*}
    for all $s,t\in [0,1]$.
\end{defn}

\begin{thm}\label{thm:Bottleneck is intrinsic to erosion}
    The intrinsic metric corresponding to erosion distance  $\hat{d}_E$ is bottleneck distance $d_B$. That is,
    \[
    \hat{d}_E(Y,Y') = d_B(Y,Y')
    \]
    for all $Y,Y'\in \mathsf{PDgm}$.
\end{thm}

\begin{proof}
Let $Y,Y'\in\mathsf{PDgm}$ be two persistence diagrams and let $\alpha$ be a continuous path in $\mathsf{PDgm}$ from $Y$ to $Y'$, where continuity is with respect to the topology induced by $d_E$ (or $d_B$, which induces the same topology, by Theorem \ref{thm: bottleneck is erosion}). We will show that $\alpha$ is rectifiable with respect to $d_E$ if and only if it is rectifiable with respect to $d_B$, and, moreover, that 
\begin{equation}\label{eqn:length_equality}
\mathrm{len}_{d_E}(\alpha)=\mathrm{len}_{d_B}(\alpha).
\end{equation}
Assuming that this holds, we obtain the desired result:
    \[
    \hat{d}_{E}(Y,Y') = \inf \mathrm{len}_{d_E}(\alpha) = \inf \mathrm{len}_{d_B}(\alpha) = d_B(Y,Y'),
    \]
    where the last equality follows from the fact that $(\mathsf{PDgm},d_B)$ is geodesic (see, e.g., \cite{chowdhury2019geodesics,che2024basic}). 

First, assume that $\alpha:[0,1] \to \mathsf{PDgm}$ is rectifiable with respect to $d_B$. For any partition $0 = t_0 < t_1 < \cdots < t_{m+1} = 1$, 
\[
\sum_{i=0}^m d_E(\alpha(t_i),\alpha(t_{i+1})) \leq \sum_{i=0}^m d_B(\alpha(t_i),\alpha(t_{i+1})),
\]
so the fact that the supremum over partitions for the sum on the right hand side is finite implies $\alpha$ is also rectifiable with respect to $d_E$. It remains to prove the converse, and the equality \eqref{eqn:length_equality}.

Now let $\alpha$ be a path which is rectifiable with respect to $d_E$ and set $L \coloneqq \mathrm{len}_{d_E}(\alpha)$, assuming without loss of generality that $L > 0$ (otherwise the conclusion is obvious). Parameterizing by arclength, we can assume that $\alpha:[0,L] \to \mathsf{PDgm}$ is 1-Lipschitz with respect to erosion distance:
\[
d_E(\alpha(s),\alpha(t))\leq |s-t|.
\]
To establish \eqref{eqn:length_equality}, the first step is to show that $\alpha$ is also Lipschitz with respect to bottleneck distance. We will refer to a partition
\[
a=t_0<t_1< \cdots < t_{m+1} = b
\]
of an interval $[a,b]\subset [0,L]$ as a \textbf{good partition} if
\[
d_B(\alpha(t_i),\alpha(t_{i+1})) \leq 3 \cdot d_E(\alpha(t_i),\alpha(t_{i+1}))
\]
for all $0\leq i\leq m$. 

We claim that any interval $[a,b] \subset [0,L]$ with $a<b$ admits a good partition. To see this, first recall that, by Theorem \ref{thm: bottleneck is erosion}, for any
$t\in[0,L]$, there exists a $d_E$-neighborhood $U_t$ of $Y_t \coloneqq \alpha(t)$ in which
\[
d_B(Y_t,Y'')\leq 3 \cdot d_E(Y_t,Y'')
\]
for every $Y'' \in U_t$. By continuity, there exists $\delta_t>0$ such that
\[
d_B(Y_t,Y_s)\leq 3 \cdot d_E(Y_t,Y_s)
\]
whenever $s\in[0,L]$ and $|s-t|<\delta_t$. Next, let
\[
T=\sup\{u\in(a,b]\mid [a,u]\text{ admits a good partition}\}.
\]
Choosing $u\in(a,b]$ sufficiently close to $a$ that $u-a<\delta_a$, the partition $\{a,u\}$ is good. Thus the set is nonempty, so that $T$ is well defined. Choose $u$ in this set so that $T-u < \delta_T$.  If $u<T$, appending $T$ to a good partition of $[a,u]$ gives a good partition of $[a,T]$; if $u=T$, such a partition already exists. If $T<b$, the local bound at $T$ allows us to append a point $v\in(T,b]$ sufficiently close to $T$, contradicting the definition of $T$. Therefore $T=b$, and $[a,b]$ admits a good partition. 

The existence of good partitions now gives the desired 3-Lipschitz result. For any $0 \leq s < t \leq L$, choose a good partition $s = t_0 < t_1 < \cdots < t_{m+1} = t$. Then the triangle inequality and the fact that $\alpha$
is $1$-Lipschitz with respect to $d_E$ give
\begin{align*}
d_B(\alpha(s),\alpha(t)) \leq
\sum_{i=0}^{m}
d_B(\alpha(t_i),\alpha(t_{i+1})) \leq
3\sum_{i=0}^{m}
d_E(\alpha(t_i),\alpha(t_{i+1})) \leq
3\sum_{i=0}^{m}(t_{i+1}-t_i)
=3(t-s).
\end{align*}
This proves the 3-Lipschitz claim. Moreover, a Lipschitz curve is rectifiable, so $\alpha$ is rectifiable with respect to $d_B$, as claimed.

Finally, we prove the equality \eqref{eqn:length_equality}. Since $\alpha$ is Lipschitz with respect to both distances $d_E$ and $d_B$,
the metric derivatives
\[
\dot\alpha_{d_E}(t)
\coloneqq
\lim_{h\to 0}\frac{d_E(Y_{t+h},Y_t)}{|h|}
\qquad\text{and}\qquad
\dot\alpha_{d_B}(t)
\coloneqq
\lim_{h\to 0}\frac{d_B(Y_{t+h},Y_t)}{|h|}
\]
exist for almost every $t\in(0,L)$; this follows from, e.g., ~\cite[Theorem 2.7.6]{burago2001course}, which also says that 
\[
\mathrm{len}_{d_E}(\alpha)
=\int_0^L\dot\alpha_{d_E}(t)\,dt,
\qquad
\mathrm{len}_{d_B}(\alpha)
=\int_0^L\dot\alpha_{d_B}(t)\,dt.
\]
It therefore suffices to show that $\dot\alpha_{d_E} = \dot\alpha_{d_B}$ holds almost everywhere.

For each integer $n\geq 0$, define
\[
A_n\coloneqq\{t\in[0,L]\mid |Y_t|=n\},
\]
where the cardinality $|Y_t|$ is counted with multiplicity. As it is a subset of the real line, each $A_n$ has at most
countably many isolated points. For almost every $t \in (0,L)$,  $\dot\alpha_{d_E}(t)$ and $\dot\alpha_{d_B}(t)$ both exist and $t$ is not an isolated point in $A_{|Y_t|}$; fix such a point $t$.  Then there exists a
sequence of points $t_j\neq t$ converging to $t$ such that $|Y_{t_j}|=|Y_t|$ for every $j$. By continuity and the equal-cardinality case of Theorem \ref{thm: bottleneck is erosion}, we have
\[
d_E(Y_{t_j},Y_t)=d_B(Y_{t_j},Y_t)
\]
for all sufficiently large $j$. Taking the limits along this sequence gives
\[
\dot\alpha_{d_E}(t) = \lim_{j\to\infty} \frac{d_E(Y_{t_j},Y_t)}{|t_j-t|} = \lim_{j\to\infty} \frac{d_B(Y_{t_j},Y_t)}{|t_j-t|} = \dot\alpha_{d_B}(t).
\]
This completes the proof.
\end{proof}

We have two immediate corollaries.

\begin{cor} \label{cor: d_E is not geodesic}
    The space $(\mathsf{PDgm},d_E)$ is not a length space.
\end{cor}

\begin{proof}
As was observed in Section \ref{subsec:erosion_distance_background}, $d_E \neq d_B$, in general. Since the intrinsic metric of $d_E$ is not itself, $(\mathsf{PDgm},d_E)$ is not a length space. 
\end{proof}

Applying Theorem \ref{thm: d_E = Landscape Distance} yields the following reinterpretation.

\begin{cor} \label{cor: d_L is not geodesic}
    The space $(\mathsf{PDgm},d_L)$ is not a length space. In fact, the bottleneck distance is the intrinsic metric associated to the landscape distance.
\end{cor}

\subsection{Birth-Zero Diagrams}

We study a specific subset of persistence diagrams, defined as follows.

\begin{defn} \label{defn: birth-zero persistence diagrams}
    By $\mathsf{PDgm}^{0}$, we denote the proper subset of persistence diagrams of the form 
    \begin{align*}
        Y = \{ (0,d_i)\}_{i=1}^N
    \end{align*}
    with $d_1\geq d_2\geq \cdots \geq d_N >0$. We refer to such diagrams as \textbf{birth-zero persistence diagrams}.
\end{defn}

\begin{remark}
    When one works with persistence diagrams that arise from the Vietoris-Rips filtration in homology degree zero $H_0$, the resulting diagram is a birth-zero persistence diagram. For this reason, it is sometimes referred to as a \emph{zero-degree persistence diagram}, as in \cite[Patrangenaru et al., 2019]{patrangenaru2019challenges}.
\end{remark}
On this subspace of persistence diagrams, we are able to derive a stronger connection between the bottleneck and erosion distances, and to relate them to an alternative vectorization introduced by Patrangenaru et al.\, in \cite{patrangenaru2019challenges}. These results provide key tools for addressing coarse non-embeddability questions in Section \ref{sec:coarse_embeddings} below. 

Although Proposition \ref{prop:bottleck_and_erosion_not_coarsely_equivalent} states that bottleneck and erosion distances are not equivalent (even in a weak sense) on the full space $\mathsf{PDgm}$, the next result shows that they agree on the subspace $\mathsf{PDgm^0}$. 

\begin{thm} \label{thm: d_B = d_E on PDgm^0}
The bottleneck distance and the erosion distance agree on $\mathsf{PDgm}^0$; that is, for any $Y, Y' \in \mathsf{PDgm}^0$, we have 
\[
    d_B(Y, Y') = d_E(Y, Y').
\]
\end{thm}

\begin{proof}
    First consider diagrams consisting of single points $\{(0,d)\}$ and $\{(0,d')\}$. It is not hard to show, by direct computation, that
    \[
    d_L(\{(0,d)\},\{(0,d')\}) = \|f_{(0,d)} - f_{(0,d')}\|_\infty = \min \{|d-d'|,\max\{d/2,d'/2\}\}.
    \]
    On the other hand, this is also the formula for $d_B(\{(0,d)\},\{(0,d')\})$ (there are only two choices of matchings for these diagrams, whose costs are the elements of the set given above).

    Now consider the general case $Y = \{(0,d_i)\}_{i=1}^N$ and $Y' = \{(0,d_i')\}_{i=1}^{N'}$. Without loss of generality, assume that $N' \geq N$ and set $d_{N+1}=\cdots = d_{N'} = 0$. Then, by the fact observed above and Theorem \ref{thm: d_E = Landscape Distance},
    \[
    \begin{split}
    d_E(Y,Y') = d_L(Y,Y') &= \max_{i=1,\ldots,N'} d_L(\{(0,d_i)\},\{(0,d_i')\}) \\
    &= \max_{i=1,\ldots,N'} d_B(\{(0,d_i)\},\{(0,d_i')\}) \geq d_B(Y,Y'),
    \end{split}
    \]
    where the last inequality follows because the quantity on its left-hand side is the cost of a particular matching. Indeed, it is easy to verify that
    \[
    \mathrm{cost}(\phi) = \max_{i=1,\dots,N'} \min\{|d_i-d'_i|,\max\{d_i/2,d'_i/2\}\}
    \]
    where the matching $\phi$ pairs $(0,d_i)\overset{\phi}{\longleftrightarrow} (0,d'_i)$ if and only if
    \[
        |d_i-d'_i| \leq \max\{d_i/2,d'_i/2\}
    \]
    and otherwise both points are matched with the diagonal, i.e.,
    \[
        (0,d_i)\overset{\phi}{\longleftrightarrow} \Delta \quad \text{ and } \quad (0,d'_i)\overset{\phi}{\longleftrightarrow} \Delta
    \]
    for each $i = 1,\dots,N$. Moreover, for indices $i>N$, the points $(0,d'_i)$ are matched with $\Delta$. Since $d_E \leq d_B$, in general, this completes the proof.
\end{proof}

\begin{remark}
    A result similar to Theorem \ref{thm: d_B = d_E on PDgm^0}, comparing bottleneck distance and a certain interleaving distance between Betti curves on the subspace of birth-zero persistence diagrams, is obtained in \cite[Theorem 4.14]{kim2021spatiotemporal}. The proof strategies are similar, but comparing the explicit calculations of \cite[Example 4.15]{kim2021spatiotemporal} shows that their result is distinct from ours.
\end{remark}

\begin{remark}\label{rmk:bottleneck_distance_computation}
    The computational complexity of the bottleneck distance was initially explored by Efrat et al. ~\cite{efrat2001geometry}, and more recently Kerber et al.~\cite[Theorem 3.1]{10.1145/3064175} proposed an efficient algorithm. While we do not delve into a detailed discussion of the bottleneck distance's computational complexity here, it is worth noting that the proof of Theorem \ref{thm: d_B = d_E on PDgm^0} leads to a simple formulation---which suggests a linear-time computation---of the distances $d_B = d_E = d_L$ on $\mathsf{PDgm}^0$, as we record in the following corollary. Indeed, the partial matching constructed in the proof of Theorem \ref{thm: d_B = d_E on PDgm^0} is optimal.
\end{remark}

\begin{cor}\label{cor:formula_bottleneck_degree_0}
    Let $Y = \{(0,d_i)\}_{i=1}^N$ and $Y' = \{(0,d_i')\}_{i=1}^{N'}$, and without loss of generality, assume that $N' \geq N$. Set $d_{N+1}=\cdots = d_{N'} = 0$. For $d_\bullet = d_B = d_E = d_L$, we have
    \[
    d_\bullet(Y,Y') = \max_{i=1,\dots,N'} \min \bigg\{ |d_i - d'_i|, \max\{ d_i/2,d'_i/2 \} \bigg\}.
    \]
\end{cor}

Birth-zero persistence diagrams can be represented as sequences in $\ell^\infty$ by listing their death coordinates followed by infinitely many zeros. This is known as the \textbf{death-vectorization} embedding due to Patrangenaru et al. \cite{patrangenaru2019challenges}. An easy corollary of Theorem \ref{thm: d_B = d_E on PDgm^0} is that this embedding is bi-Lipschitz. Let us now formalize this result. 

\begin{defn} \cite{patrangenaru2019challenges}
The \textbf{death-vectorization} map $\mathrm{DV} : \mathsf{PDgm}^0 \to \ell^\infty$ is defined by 
\[
\mathrm{DV} \left( \left\{ (0, d_i) \right\}_{i=1}^n \right) = (d_1, \dots, d_n, 0, \dots).
\]
\end{defn}

\begin{remark} \label{rmk: image of the death-vectorization map}
Let $\ell^\infty_\geq \subset \ell^\infty$ denote the subset of sequences of the form 
\[
(a_1, \dots, a_n, 0, \dots)
\]
with $a_1 \geq \dots \geq a_n > 0$ for some $n \in \N$. Notice that this is the image of the death-vectorization map, and that $\mathrm{DV} : \mathsf{PDgm}^0 \to \ell^\infty_{\geq}$ is a bijection.
\end{remark}

\begin{cor} \label{cor: Death vectorization map is bilipschitz}
The death-vectorization map satisfies
\[
d_\bullet(Y, Y') \leq \| \mathrm{DV}(Y) - \mathrm{DV}(Y') \|_\infty \leq 2 \, d_\bullet(Y, Y')
\]
for all $Y, Y' \in \mathsf{PDgm}^0$, where $d_\bullet = d_L = d_B = d_E$. 
\end{cor}

\begin{proof}
    Let $Y = \{(0,d_i)\}_{i=1}^N$ and $Y' = \{(0,d_i')\}_{i=1}^{N'}$ be an arbitrary pair of birth-zero diagrams, as stated in Corollary \ref{cor:formula_bottleneck_degree_0}. Then,
    \begin{align*}
        d_\bullet(Y,Y') = \max_{i=1,\dots,N'} \min \{ |d_i - d'_i|, \max\{ d_i/2,d'_i/2 \} \} &\leq \max_{i=1,\dots,N'} |d_i - d'_i| \\ &= \| \mathrm{DV}(Y) - \mathrm{DV}(Y')\|_{\infty}
    \end{align*}
    is obvious. On the other hand,
    \begin{align*}
         2 \, d_\bullet(Y,Y')= \max_{i=1,\dots,N'} \min \{2 \,|d_i - d'_i|, \max\{ d_i,d'_i \} \} &\geq \max_{i=1,\dots,N'} |d_i - d'_i| \\&= \| \mathrm{DV}(Y) - \mathrm{DV}(Y')\|_{\infty}.
    \end{align*}
    Indeed, if $2\,d_\bullet(Y,Y')$ is realized by $2\,|d_i - d_i'|$ then the inequality is clear; if it is realized by, say, $d_i$, then $d_i \geq d_i'$, so that $|d_i-d_i'| = d_i - d_i' \leq d_i$.
\end{proof}

\begin{remark}\label{rmk:tightness}
   It is not hard to find examples which illustrate that the bounds in Corollary \ref{cor: Death vectorization map is bilipschitz} are tight. Indeed, take $Y = \{(0,1)\}$, $Y' = \{(0,\varepsilon)\}$ for some small $\varepsilon > 0$. Then $d_L(Y,Y') = 1/2$, while the death vectorization distance is $1-\varepsilon$. Taking $\varepsilon \to 0$ gets arbitrarily close to realizing the upper bound. Likewise, it is easy to find an example showing that the lower bound is tight.
\end{remark}

\begin{remark}\label{rmk:relation_to_memoli_zhou}
It is shown in \cite{memoli2022ephemeral,memoli2023ephemeral} that if $Y$ and $Y'$ are diagrams derived from degree-0 persistent homology of the Vietoris-Rips complexes of finite metric spaces $X$ and $X'$, respectively, then
\[
d_B(Y,Y') \leq \| \mathrm{DV}(Y) - \mathrm{DV}(Y') \|_\infty \leq 2 \cdot d_{GH}(X,X'),
\]
where $d_{GH}$ is the \emph{Gromov-Hausdorff distance} between metric spaces (see \cite{burago2001course}). Indeed, this follows from the main stability result \cite[Theorem 5]{memoli2023ephemeral} together with \cite[Proposition 1.5]{memoli2022ephemeral} (from the extended arXiv version of the paper), which specializes to the birth-zero setting. 

By the general Gromov-Hausdorff stability of persistent homology~\cite[Theorem 3.1]{chazal2009gromov}, our Theorem \ref{cor: Death vectorization map is bilipschitz} only yields the upper Gromov-Hausdorff bound
\[
\| \mathrm{DV}(Y) - \mathrm{DV}(Y') \|_\infty \leq 4 \cdot d_{GH}(X,X'),
\]
and is therefore suboptimal. On the other hand, our theorem is purely at the level of diagrams, and the example of Remark \ref{rmk:tightness} shows that the bound cannot be improved at that level. 
\end{remark}

\subsection{Coarse Non-Embeddability}\label{sec:coarse_embeddings}

Bubenik and Wagner \cite{bubenik2020embeddings} showed that persistence diagrams equipped with the bottleneck distance do not admit a coarse embedding into a Hilbert space. Motivated by their work, we strengthen their non-embeddability result by proving that even the space of birth-zero persistence diagrams--whether equipped with the bottleneck distance or the erosion distance--fails to admit a coarse embedding into a Hilbert space. This is established in Theorem \ref{thm: space of persistence diagrams do not admit coarse embeddings}, which not only offers an alternative proof of the result by Bubenik and Wagner ~\cite{bubenik2020embeddings}, but also expands upon it by extending the analysis to include the erosion and landscape distances.

\begin{thm}[Coarse Non-Embeddability in Hilbert Spaces]\label{thm: space of persistence diagrams do not admit coarse embeddings}
    The space of birth-zero persistence diagrams, $(\mathsf{PDgm}^{0},d_\bullet)$, equipped with the distance $d_\bullet = d_B = d_E = d_L$, does not admit a coarse embedding into a Hilbert space. Consequently, the full spaces of persistence diagrams $(\mathsf{PDgm},d_\bullet)$, with $d_\bullet = d_B$ or $d_\bullet =d_E=d_L$, also fail to admit such an embedding.
\end{thm}

To prove Theorem \ref{thm: space of persistence diagrams do not admit coarse embeddings}, we show that the image $\ell^\infty_\geq$ of the death-vectorization embedding (see Remark \ref{rmk: image of the death-vectorization map}) does not coarsely embed into a Hilbert space. This follows from Proposition \ref{prop: isometric copy of finite subsets in l infinity subspace} which implies that $\ell^\infty_\geq$ contains an isometric copy of every finite metric space. By a result of Mitra and Virk \cite[Corollary 4.7.]{mitra2021space}, the presence of such isometric copies constitutes an obstruction to coarse embeddability into a Hilbert space.

\begin{remark}
    The following proposition is similar to \cite[Lemma 9]{wagner2021nonembeddability}, where Wagner considers the space of persistence diagrams equipped with the $p$-Wasserstein distance for $p>2$. In contrast, we work with the space $\ell_\geq^\infty$, and the constant $c$ in our proof is different.
\end{remark}

\begin{prop}\label{prop: isometric copy of finite subsets in l infinity subspace}
    Let $d\in \N$. Every finite subset of $(\R^d,\|\cdot \|_{\infty})$ admits an isometric embedding into $(\ell_\geq^\infty, \| \cdot \|_{\infty})$.
\end{prop}

\begin{proof}
    Let $A = \{a^1,a^2,\dots,a^n\}$ be a finite subset of $(\R^d,\|\cdot \|_{\infty})$. Choose a constant $c>0$ such that
    \[
    c>\max \{\|a^i\|_{\infty},|a^i_k-a^j_{k'}|:1\leq i,j\leq n, 1\leq k,k'\leq d \}
    \]
    Define the map
    \[
        \varphi:A\to \ell^\infty_\geq,\qquad
         a^i \mapsto (2c(d+2-k) + a^i_k)_{k=1}^d
    \]
    Here, we abuse notation and write $(2c(d+2-k) + a^i_k)_{k=1}^d$ to denote the sequence $(2c(d+2-k) + a^i_k)_{k=1}^d$ extended by appending infinitely many zeros. We first claim that $\varphi(a^i) \in \ell^\infty_\geq$ for each $a^i$. Observe that $a^i_k + 2c \geq a^i_{k+1}$ for each $k$ because
    \[
    a^i_k + 2c = a^i_k + c + c > 0 + c = c> a^i_{k+1}
    \]
    and it follows that $\varphi(a^i) \in \ell^\infty_\geq$ since
    \[
    \varphi(a^i) \in \ell^\infty_\geq \iff 2c(d+2-k) + a^i_k \geq 2c(d+2 - (k+1)) + a^i_{k+1} \iff a^i_k + 2c \geq a^i_{k+1}
    \]
    Now we show that $\varphi$ is an isometry. Let $a^i$ and $a^j$ be arbitrary in $A$. Then
    \begin{align*}
    \| \varphi(a^i) - \varphi(a^j) \|_\infty &= \max_{1\leq k \leq d} \left|2c(d+2 -k) + a^i_k - 2c(d+2-k) - a^j_k\right| \\
    &= \max_{1\leq k \leq d} \left|a^i_k - a^j_k\right| = \| a^i - a^j \|_{\infty},
    \end{align*}
    which concludes the proof.
\end{proof}

\begin{cor} \label{cor: ell infty subspace does not CE}
    The space $(\ell^\infty_\geq, \| \cdot \|_{\infty})$ does not admit a coarse embedding into a Hilbert space.
\end{cor}

\begin{proof}
    We begin by noting that every finite metric space $(X =\{x_1,x_2,\dots,x_d\},d_X)$ isometrically embeds as a finite subset of $(\R^d,\|\cdot\|_\infty)$ via the map $x \mapsto (d_X(x,x_i))_{i=1}^d$, known as the \emph{Kuratowski embedding} \cite{kuratowski1935quelques} (see also \cite[Remark 1.7]{heinonen2003geometric}). Thus, Proposition \ref{prop: isometric copy of finite subsets in l infinity subspace} implies that $(\ell^\infty_\geq, \| \cdot \|_{\infty})$ contains an isometric copy of every finite metric space. By \cite[Corollary 4.7]{mitra2021space}, this implies that $(\ell^\infty_\geq, \| \cdot \|_{\infty})$ does not coarsely embed into any Hilbert space.
\end{proof}

\begin{proof}[Proof of Theorem \ref{thm: space of persistence diagrams do not admit coarse embeddings}]
    By Corollary \ref{cor: ell infty subspace does not CE}, the space $(\ell^\infty_\geq,\|\cdot\|_\infty)$ does not admit a coarse embedding into a Hilbert space. Because the spaces $(\mathsf{PDgm}^0,d_B)$ and $(\ell^\infty_\geq,\|\cdot\|_\infty)$ are bi-Lipschitz equivalent via the death-vectorization map (Corollary \ref{cor: Death vectorization map is bilipschitz}), $(\mathsf{PDgm}^0,d_B)$ also fails to admit such an embedding. Since the bottleneck and erosion distances are equivalent on $\mathsf{PDgm}^0$ (Theorem \ref{thm: d_B = d_E on PDgm^0}), the same conclusion holds for $(\mathsf{PDgm}^0,d_E)$. Consequently, the full spaces $(\mathsf{PDgm},d_B)$ and $(\mathsf{PDgm},d_E)$ also do not admit coarse embeddings into a Hilbert space. 
\end{proof}

 \section*{Acknowledgements} We would like to thank Parker Edwards, Woojin Kim and Facundo M\'{e}moli for useful conversations and feedback which helped us to improve the paper. This research was partially supported by NSF grant DMS--2324962.
 
 \section*{AI Declaration} ChatGPT 6 Astra was used to proofread and revise the manuscript. In particular, the model found a gap in the original argument for Theorem \ref{thm: bottleneck is erosion}, whose statement in previous versions of the paper omitted the equal cardinality assumption. The model also suggested a strategy for repairing the proof, and for necessary downstream repairs to the proof of Theorem \ref{thm:Bottleneck is intrinsic to erosion}, which were both adapted to the proofs presented here. All text in the manuscript was written by the authors, and the authors take full responsibility for the mathematical content of the article.

\bibliographystyle{plainnat}
\bibliography{References}

\end{document}